\documentclass[11pt]{article}

\usepackage{amsfonts, amsmath, amsthm, esint,graphicx, color}

\newcommand{\R}{\mathbb{R}}

\newcommand{\N}{\mathbb{N}}
\newcommand{\Ss}{\mathbb{S}}
\newcommand{\Z}{\mathbb{Z}}
\newcommand{\A}{\mathcal{A}}
\newcommand{\E}{\mathcal{E}}
\newcommand{\set}[2]{\left\{#1 \colon #2\right\}}
\newcommand{\dd}[2]{\frac{\partial #1}{\partial #2}}

\newcommand{\loc}{\mathrm{loc}}
\newcommand{\blank}{{\mkern 2mu\cdot\mkern 2mu}}
\newcommand{\scp}[2]{\left\langle #1, #2 \right\rangle}

\newcommand{\eps}{\varepsilon}

\newcommand{\be}{\begin{equation}}
\newcommand{\ee}{\end{equation}}

\renewcommand{\div}{\operatorname{div}}

\DeclareMathOperator{\curl}{curl}
\DeclareMathOperator{\supp}{supp}
\DeclareMathOperator*{\osc}{osc}
\DeclareMathOperator{\diam}{diam}

\newtheorem{theorem}{Theorem}
\newtheorem{lemma}[theorem]{Lemma}
\newtheorem{corollary}[theorem]{Corollary}
\newtheorem{proposition}[theorem]{Proposition}
\newtheorem{definition}[theorem]{Definition}
\newtheorem{remark}[theorem]{Remark}

\newcounter{stepno}
\newenvironment{severalsteps}{\setcounter{stepno}{0}}{}
\newcommand{\step}[1]{\refstepcounter{stepno} \par \smallskip \noindent \textit{Step \arabic{stepno}: #1.} }

\newcounter{caseno}

\newenvironment{severalcases}{\setcounter{caseno}{0}}{}
\newcommand{\case}{\refstepcounter{caseno} \par \smallskip \noindent \textit{Case \arabic{caseno}:} }

\title{Energy minimisers of prescribed winding number in an $\Ss^1$-valued nonlocal Allen-Cahn type model}

\author{
{\Large Radu Ignat}\footnote{Institut de Math\'ematiques de Toulouse \& Institut Universitaire de France, UMR 5219, Universit\'e de Toulouse, CNRS, UPS
IMT, F-31062 Toulouse Cedex 9, France. Email: Radu.Ignat@math.univ-toulouse.fr} 
\and {\Large Roger Moser}\footnote{Department of Mathematical Sciences,
University of Bath,
Bath BA2 7AY,
UK. Email: r.moser@bath.ac.uk}
}

\begin{document}

\maketitle

\begin{abstract}
We study a variational model for transition layers in thin ferromagnetic films with an underlying functional that
combines an Allen-Cahn type structure with an additional nonlocal
interaction term. The model represents the magnetisation by a map from $\R$ to $\Ss^1$.
Thus it has a topological invariant in the
form of a winding number, and we study minimisers subject to a
prescribed winding number. As shown in our previous paper \cite{Ignat-Moser:17},
the nonlocal term gives rise to solutions that would not be present
for a functional including only the (local) Allen-Cahn terms.
We complete the picture here by proving existence of minimisers in all cases where
it has been conjectured. In addition, we prove non-existence in some other cases.
\end{abstract}


\section{Introduction}

In this paper we study a variational model coming from the theory of
micromagnetics. In soft thin films of ferromagnetic materials, one of the predominant
structures in the magnetisation field is a type of transition layer, called a N\'eel wall.
We consider a simplified, one-dimensional variational
model for N\'eel walls and study the question whether several transitions may
combine to form a more complex transition layer.
The same model has been used by several authors to analyse N\'eel walls in terms of
existence, uniqueness, and properties of solutions to the Euler-Lagrange equations,
but mostly for single transitions (see Section \ref{sec:neel} for more details).
For the background and the derivation of the model, we refer to
the papers
\cite{DKMOreduced, DKMO05, HS98}.

\subsection{The model} We begin with a description of the variational model used for our theory.
For a given parameter $h \in [0, 1]$, consider maps $m=(m_1, m_2):\R \to \Ss^1$ with values on the unit circle $\Ss^1$.
We study the functional
\[
E_h(m) = \frac{1}{2} \left(\|m'\|^2_{L^2(\R)} + \|m_1-h\|^2_{\dot{H}^{1/2}(\R)} + \|m_1-h\|^2_{L^{2}(\R)}\right),
\]
where $m'$ is the derivative of $m$ and $\dot{H}^{1/2}(\R)$ denotes the homogeneous Sobolev space of order $1/2$ (a different representation of
the corresponding term is given shortly). The first two terms in this functional represent what is called
the exchange energy and the stray field energy in the full micromagnetic model, and we will use these expressions
here as well. The third term comes from a combination of crystalline anisotropy and an external magnetic
field. For simplicity, we call this term the anisotropy energy.

The stray field energy $\|m_1-h\|^2_{\dot{H}^{1/2}(\R)}$ arises from the micromagnetic theory in conjunction
with a stray field potential $u \colon \R_+^2 \to \R$ (where $\R_+^2 = \R \times (0, \infty)$), which
solves
\begin{alignat}{2}
\Delta u & = 0 & \quad & \text{in $\R_+^2$}, \label{eqn:u_harmonic_higher} \\
\dd{u}{x_2} & = -m_1' && \text{on $\R \times \{0\}$}. \label{eqn:boundary_condition_higher}
\end{alignat}
This boundary value problem has a unique solution up to constants if we impose finite Dirichlet energy.
(We will discuss this point in more detail in Section \ref{sect:prelim}.)
The solution then satisfies
\[
\int_{\R_+^2} |\nabla u|^2 \, dx=\|m_1-h\|^2_{\dot{H}^{1/2}(\R)}.
\]
(The constant $h$ may seem irrelevant here, because $\|\blank\|_{\dot{H}^{1/2}(\R)}$ is a
seminorm that vanishes on constant functions. Notwithstanding, we will keep $h$ in the expression
as a reminder that $m_1 - h$ decays to $0$ at $\pm \infty$ for the profiles we are interested in.)
For some of the arguments in this paper, however, the following double
integral representation is more convenient (see, e.g., \cite{DiNezza-Palatucci-Valdinoci:12}):
\begin{equation} \label{eqn:double_integral}
\|m_1-h\|^2_{\dot{H}^{1/2}(\R)} = \frac{1}{2\pi} \int_{\R}\int_{\R}\frac{|m_1(x_1)-m_1(y_1)|^2}{|x_1-y_1|^2}\,dx_1 dy_1.
\end{equation}

As remarked previously, the third term in the energy functional represents, up to the constant
$h^2$, the combined effects of the
anisotropy potential $V(m) = m_1^2=1-m_2^2$ (favouring the  
easy axis $\pm e_2$) and the external field $-2h e_1$ (favouring the direction $e_1 = (1, 0)$ for positive $h$)
that gives rise to the term $-2h e_1\cdot m=-2hm_1$.
We denote
\[
\alpha = \arccos h\in [0, \pi/2],
\]
and we assume this relationship between $h$ and $\alpha$ throughout the paper.
Then the resulting potential
\[
W(m) = (m_1 - h)^2, \quad m\in \Ss^1,
\]
has two wells on the unit circle if $h \in [0,1)$,
which are at $(\cos \alpha, \pm \sin \alpha)$, and one well at $(1, 0)$ if $h=1$.
If we write $m=(\cos \phi, \sin \phi) \in \Ss^1$, then we can further observe that
$W$ grows quadratically in $\phi \pm \alpha$
near these wells if $h<1$ and quartically in $\phi$ near the well $(1,0)$ if $h=1$.

In principle, we could allow $h > 1$ as well, but 
the questions studied in this paper are
completely understood in this case 
by our previous work \cite{Ignat-Moser:17}. Since the case $h > 1$ would require
a somewhat different representation of the potential $W$, we omit the discussion here; however, 
we wish to point out that our previous paper \cite{Ignat-Moser:17} also \emph{partially} treats the case $h \in [0, 1)$
(in addition to $h > 1$), but not
$h = 1$, because the quartic growth of $W$ near the wells is not compatible with the methods used there.
From the physical point of view, the cases $h \ge 1$ and $h < 1$ are equally interesting, but
mathematically the latter is more interesting because it gives rise to the more intricate
patterns.

\subsection{N\'eel walls}
\label{sec:neel}

A transition of $m$ between the wells of the potential $W$ on $\Ss^1$, as illustrated in
Figure~\ref{fig:arrows}, represents a N\'eel wall (in the
micromagnetics terminology). In the case of $h=1$ a transition means that $m$ describes
a full rotation around $\Ss^1$. Thus in this case,
we have the transition angles $\pm 2\pi$, while for $1>h=\cos \alpha$, we have the possible transitions
angles $\pm 2\alpha$ and $\pm 2(\pi - \alpha)$ (see Figure
\ref{fig:arrows}).
\begin{figure}[htb!]
\centering
\begin{minipage}{.3\linewidth}
\centering
\includegraphics[width=\linewidth]{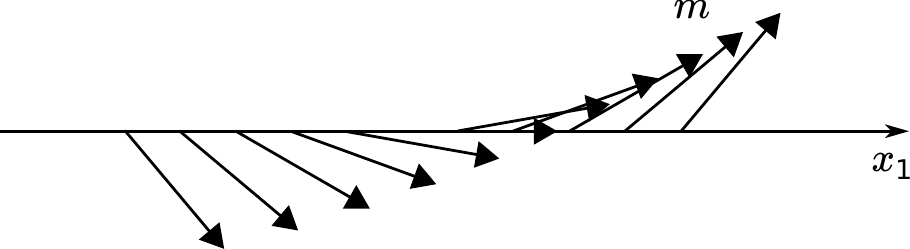}
\end{minipage}
\hspace{1cm}
\begin{minipage}{.55\linewidth}
\centering
\includegraphics[width=\linewidth]{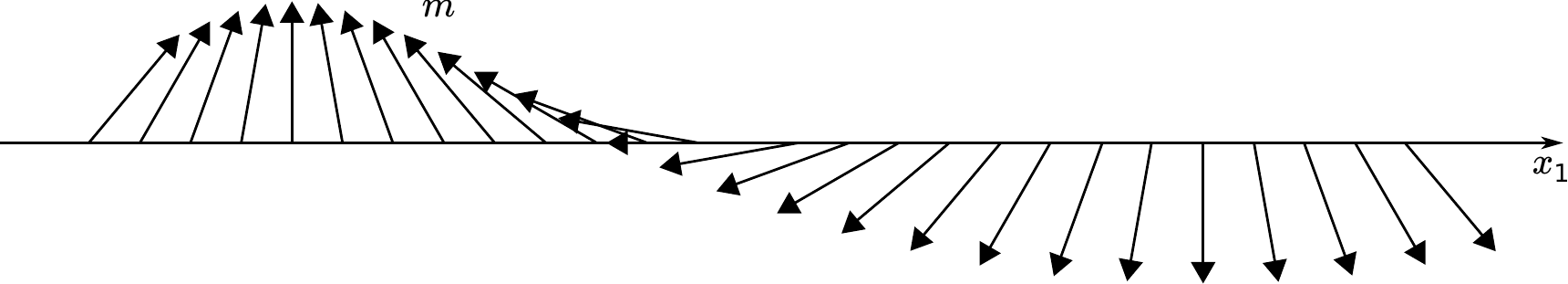}
\end{minipage}
\caption{Schematic representation of a N\'eel wall of angle $2\alpha$ (left) and $2\pi - 2\alpha$ (right).}
\label{fig:arrows}
\end{figure}
These are the most simple transitions, going from one well to the next. It is also conceivable, however,
that $m$ will pass several wells during the transition. Such behaviour is in fact necessary for profiles
of winding number $1$ or above (in the case $h < 1$), which is why we discuss the winding number {in the next section}.

Existence, uniqueness and structure of locally energy minimising profiles including a N\'eel wall of angle $2\alpha$,
for $\alpha\in (0, \frac \pi 2]$, have been proved in \cite{Cervera, Ignat-Moser:17, Me1, Me2}. In this context,
a N\'eel wall of angle $2\alpha$ is a (unique) two-scale object: it has a core of length $l \sim 1$ and two tails of larger length scale $\gg l$, where $m_1-h$ decays logarithmically. Stability, compactness, and optimality of N\'eel walls under two-dimensional perturbations have been proved in \cite{Cote-Ignat-Miot, DKO, IO}. Existence and uniqueness results are also available for N\'eel walls of larger angle $\alpha\in (\frac \pi 2, \pi)$ (see \cite{Chermisi-Muratov:13, Ignat-Moser:17}) as well as for transition layers with prescribed winding number combining several N\'eel walls (see \cite{Ignat_Knupfer, Ignat-Moser:17}).
Furthermore, the interaction between several N\'eel walls has been determined in terms of the energy (see \cite{DKMO_rep, Ignat-Moser:16}).

\subsection{Winding number}

To each finite energy configuration $m:\R \to \Ss^1$ we associate a winding number $\deg(m)$ as follows.
We first note that a map $m \in H_\loc^1(\R; \Ss^1)$ is necessarily continuous and has a continuous lifting
$\phi \colon \R \to \R$ with $m = (\cos \phi, \sin \phi)$ in $\R$. Moreover, the lifting $\phi$ is unique up to a constant.
If $E_h(m) < \infty$, then it follows that $\lim_{x_1 \to \pm \infty} m(x_1) = (\cos \alpha, \pm \sin \alpha)$
(where the signs on both sides of the equation are independent of one another).
Thus
\[
\deg(m) = \frac{1}{2\pi} \lim_{x_1 \to \infty} (\phi(x_1) - \phi(-x_1))
\]
is well-defined and belongs to $\Z \pm \{0, \alpha/\pi\}$. 

\subsection{Objective of the paper}

Our aim is to analyse the existence of minimisers $m$ of $E_h$ subject to a prescribed
winding number. For $d \in \Z \pm \{0, \alpha/\pi\}$,
we define
\[
\A_h(d) = \set{m \in  H_\loc^1(\R; \Ss^1)}{E_h(m) < \infty \text{ and } \deg(m) = d}
\]
and
\[
\E_h(d) = \inf_{m \in \A_h(d)} E_h(m).
\]
If we can find a minimiser of $E_h$ within $\mathcal{A}_h(d)$, then this will automatically be
a critical point of winding number $d$. It will in general consist of several transitions between the
points $(\cos \alpha, \pm \sin \alpha)$ on the unit circle; therefore, it can be thought of
as a composite N\'eel wall consisting of several transitions stuck together. The first component
$m_1$ of such configurations is shown schematically in Figures \ref{fig:h=1} and \ref{fig:h<1}.
\begin{figure}[htb!]
\centering
 \includegraphics[width=.6\linewidth]{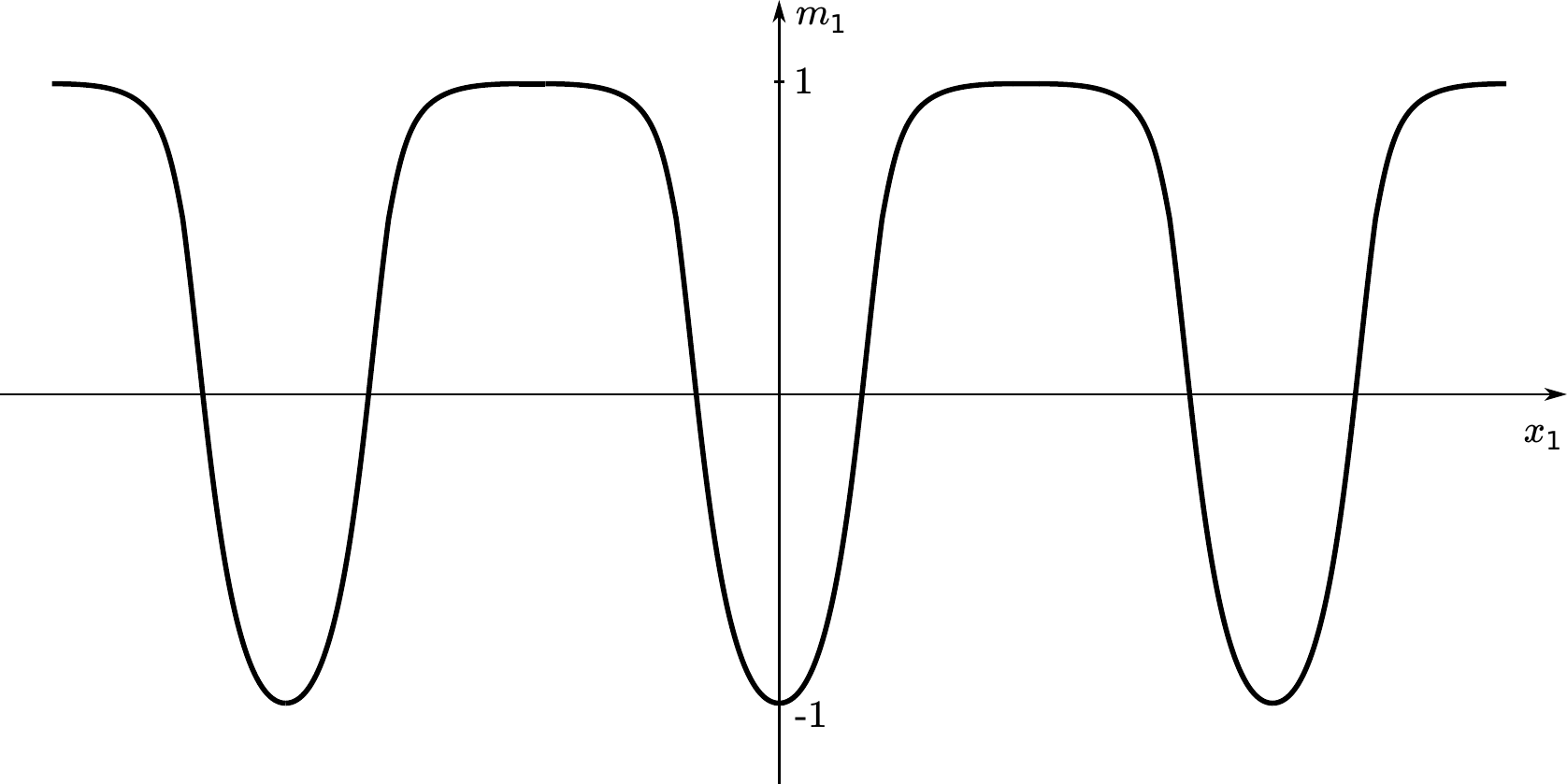}
\caption{For $h=1$, an array of N\'eel walls of total winding number $3$, represented in terms of $m_1$.}
\label{fig:h=1}
\end{figure}
\begin{figure}[htb!]
\centering
 \begin{minipage}{0.33\linewidth}
 \includegraphics[width=\textwidth]{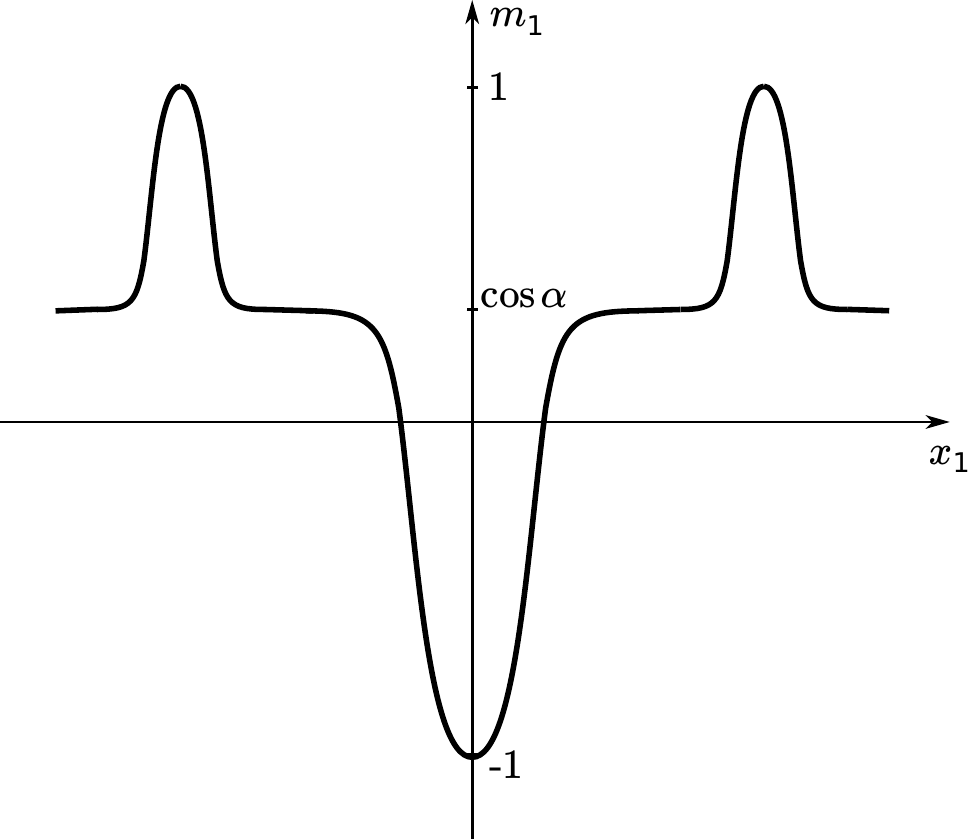}
 \end{minipage}
 \hspace{1cm}
 \begin{minipage}{0.58\linewidth}
 \includegraphics[width=\textwidth]{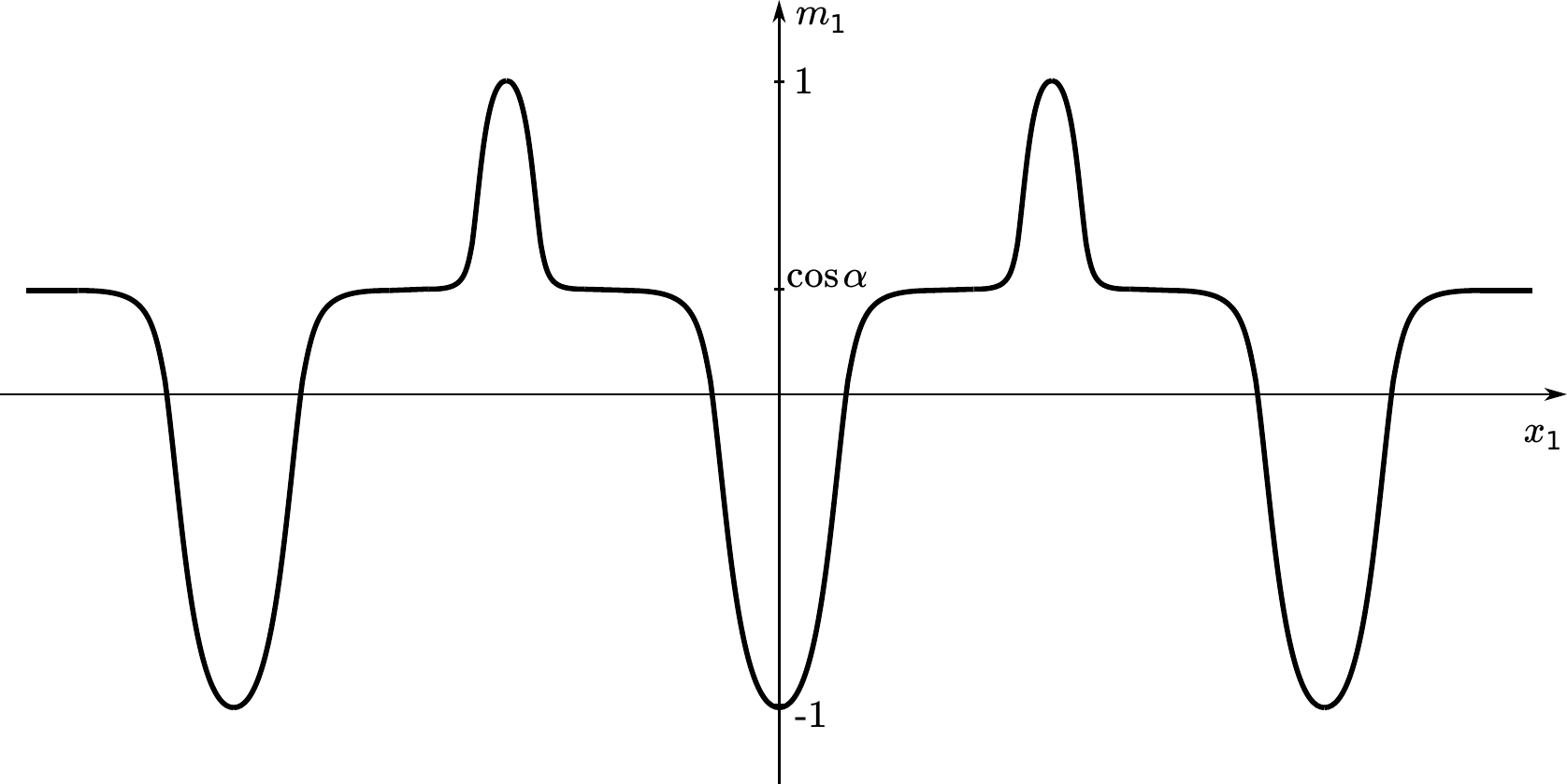}
 \end{minipage}
\caption{For $h < 1$, a hypothetical array of N\'eel walls of total winding number $1 + \alpha/\pi$ (left) and an existing one of winding number $3 - \alpha/\pi$ (right).}
  \label{fig:h<1}
\end{figure}

Clearly, there is a minimiser in $\A_h(0)$, which is constant.
Moreover, it suffices to study the question for $d \ge 0$, as the case $d < 0$
can be reduced to this one by reversing the orientation. It is well-known that a minimiser
in $\A_h(\alpha/\pi)$ exists \cite{Me1, Chermisi-Muratov:13}, and similar arguments apply to $\A_h(1 - \alpha/\pi)$
as well (see \cite{Ignat-Moser:17} for the details).
We obtained further, but still partial results on the existence of minimisers in our previous paper \cite{Ignat-Moser:17}.
Namely, there exist no minimisers in $\A_h(1)$, but if $\alpha > 0$ is sufficiently small, then
there exists a minimiser in $\A_h(2 - \alpha/\pi)$. In this paper, we
completely settle the question
under the assumption that $\alpha$ is small enough (i.e., that $h$ is sufficiently close to $1$). We also have
some information about the structure of the minimisers.

\subsection{Main results}
Our main results are as follows.

\begin{theorem}[Existence of composite N\'eel walls] \label{thm:existence}
Given $\ell \in \N = \{1,2, \dots\}$, there exists $H_\ell^- \in (0, 1)$ such that $E_h$ attains
its infimum in $\A_h(\ell - \alpha/\pi)$ for all $h \in (H_\ell^-, 1]$.
\end{theorem}

In the case $h = 1$, we note that $\ell - \alpha/\pi = \ell$. Thus the following corollary is a
special case of Theorem \ref{thm:existence}. We state it separately, because it highlights how the result fits in
with a result for $h > 1$ in our previous paper \cite{Ignat-Moser:17}. (The case $h = 1$ was not
studied in \cite{Ignat-Moser:17}, so this is new information that complements the previous results.)

\begin{corollary}
\label{cor:1}
The functional $E_1$ attains its infimum in $\A_1(\ell)$ for every $\ell \in \Z$.
\end{corollary}

\begin{theorem}[Non-existence] \label{thm:non-existence}
Given $\ell \in \N$, there exist $H_\ell^0, H_\ell^+ \in (0, 1)$ such that $E_h$ does \emph{not} attain
its infimum in $\A_h(\ell)$ for all $h \in (H_\ell^0, 1)$ and does \emph{not} attain its infimum
in $\A_h(\ell + \alpha/\pi)$ for all $h \in (H_\ell^+, 1)$.
\end{theorem}

The statements of Theorems \ref{thm:existence} and \ref{thm:non-existence}
were conjectured in our previous paper \cite{Ignat-Moser:17}, along with the conjecture that minimisers in
$\A_h(\ell - \alpha/\pi)$ do \emph{not} exist for $\ell \ge 2$ if $h$ is too small, and that the
non-existence in $\A_h(\ell)$ and in $\A_h(\ell + \alpha/\pi)$ holds for all $h \in (0, 1)$. Heuristic arguments were
provided to back up the conjectures. They rely on a decomposition of $m_1 - h$ into its positive
and negative parts and a further decomposition into pieces that correspond to individual transitions
between $(\cos \alpha, \pm \sin \alpha)$. The key observation is that the stray field energy
(the nonlocal term in the functional) will become smaller if two pieces of the same sign approach
each other or two pieces of opposite signs move away from each other. We may interpret this as
attraction between pieces of the same sign and repulsion between pieces of opposite signs.
If $1 - h$ is small, then we also expect that the positive pieces will be much smaller than the
negative pieces. Thus for winding number $\ell - \alpha/\pi$ for $\ell \in \N$ (as on the
right of Figure \ref{fig:h<1}), the whole profile will be sandwiched between the outermost pieces, which strongly
attract each other. In contrast, for the winding number $\ell + \alpha/\pi$ (as on the left of
Figure \ref{fig:h<1}), the outermost
pieces will experience a net repulsion, and moving these pieces towards $\pm \infty$ will
reduce the energy.

We summarise our results and previously known results graphically in Figure \ref{fig:existence},
alongside some conjectures from our previous paper \cite{Ignat-Moser:17}.
These conjectures would, if proved correct, complete the picture about the existence and
non-existence of minimisers of $E_h$ subject
to a prescribed winding number, except that the best values of $H_\ell^-$, $H_\ell^0$, and $H_\ell^+$
in Theorems \ref{thm:existence} and \ref{thm:non-existence} are still unknown.
(Figure \ref{fig:existence} might suggest that they are increasing in $\ell$, but no such
statement is intended and their behaviour is unknown.)
\begin{figure}[htb!]
\centering
 \includegraphics[width=.9\linewidth]{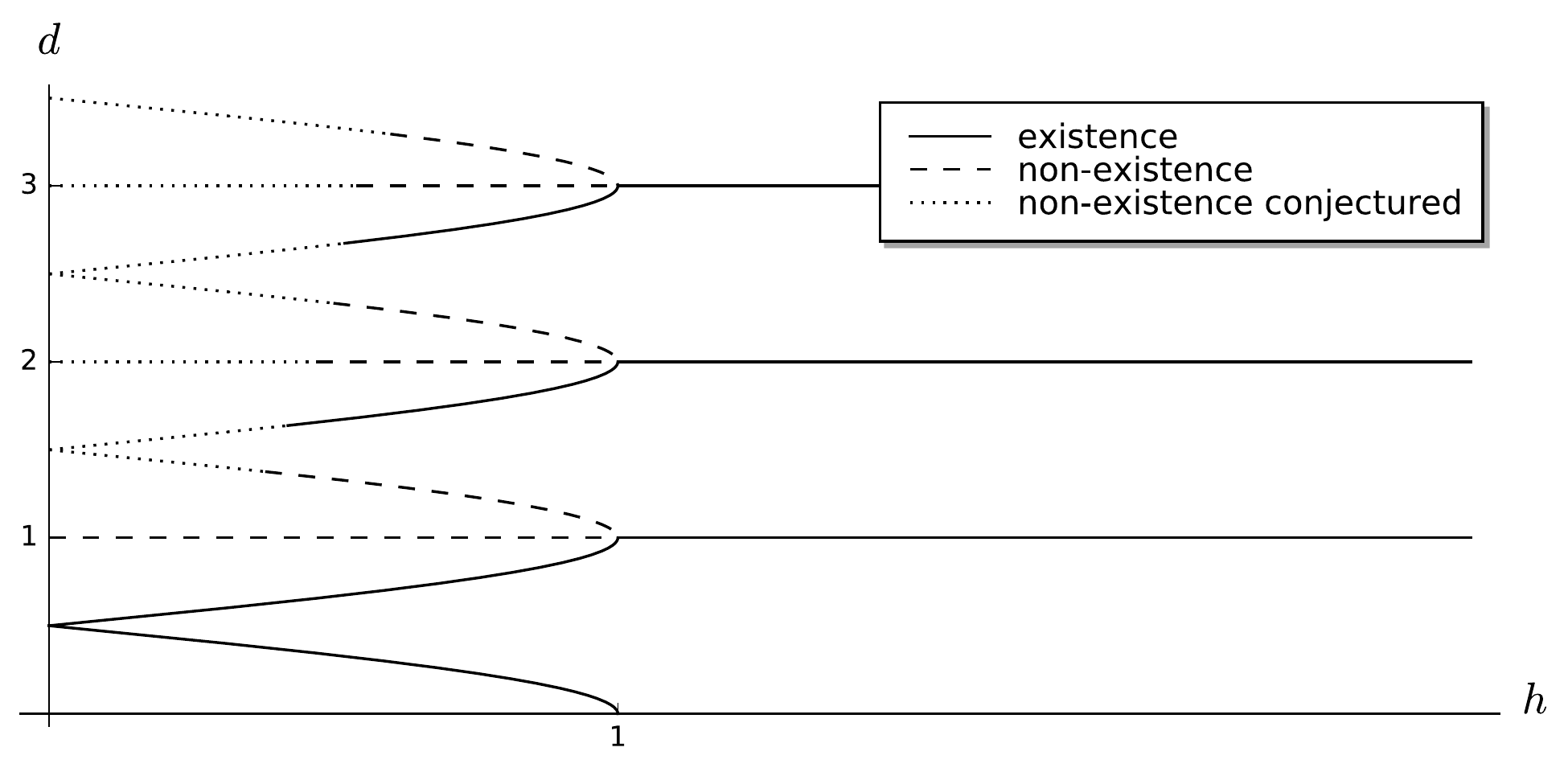}
\caption{A schematic representation of existence and non-existence results and further conjectures.
The position of any changeover between solid/dashed and dotted lines is not accurate.}
\label{fig:existence}
\end{figure}

For the proofs of Theorems \ref{thm:existence} and \ref{thm:non-existence}, we need to quantify the above heuristic arguments precisely.
Moreover, we need to estimate any effects coming from the
other (local) terms in the energy functional as well, so that we can show that the above effects (coming from the nonlocal term)
really do dominate the behaviour. In our previous paper \cite{Ignat-Moser:17}, we used
a linearisation of the Euler-Lagrange equation for minimisers of $E_h$ as one of our
principal tools. This approach is based on ideas of Chermisi-Muratov \cite{Chermisi-Muratov:13}.
It has the disadvantage, however, that it requires the quadratic growth of $W$ near
the wells that we have observed for $h < 1$ but not for $h = 1$. There are further complications
of a technical nature, and as a result, the method gives good estimates (in the case $h < 1$)
near the tails of a profile, but not between two N\'eel walls.
This in turn restricts the analysis to small winding numbers.
In this paper, we replace this tool with different arguments of variational nature.
Our new estimates are more robust; in particular they do not have the restrictions described.

In the next step, we use the estimates to show that
for certain profiles of $m$, splitting them into several parts of lower winding number will always
increase the energy. In the concentration-compactness framework of Lions \cite{Lions:84.1}, this
implies that no dichotomy will occur for minimising sequences. Here the strategy for the proof of Theorem \ref{thm:existence} is similar to our previous paper \cite{Ignat-Moser:17}.
For the non-existence, we
have exactly the opposite:
profiles of certain winding numbers can always be split into several parts in a way
that decreases the energy.

In a forthcoming paper \cite{Ignat-Moser:18} (extending previous work \cite{Ignat-Moser:16}),
we also study the asymptotic behaviour of a version of
the problem where the exchange energy is weighted with a parameter $\epsilon$ that tends to $0$.
The above heuristics are consistent with the observation that in this situation, too, the
N\'eel walls stay away from each other when $\alpha\in (0, \frac\pi 2]$ is sufficiently large. But for the
asymptotic problem, we can give the precise angle where the change in behaviour occurs.

In addition to existence, we can give some information about the structure of the
minimisers from Theorem \ref{thm:existence}.\footnote{A similar result holds for the case $h=1$, see Lemma \ref{lem:no_unaccounted_walls} below.}

\begin{theorem}[Structure] \label{thm:structure}
Given $\ell \in \N$, there exists $H_\ell \in (0, 1)$ such that the following
holds true for all $h \in (H_\ell, 1)$: if $m \in \A_h(\ell - \alpha/\pi)$ is a
minimiser of $E_h$ in $\A_h(\ell - \alpha/\pi)$, then there exist
$a_1, \dotsc, a_{2\ell - 1}, b_1, \dotsc, b_{2\ell - 2} \in \R$ with
\[
a_1 < b_1 < a_2 < \dotsb < a_{2\ell - 2} < b_{2\ell - 2} < a_{2\ell - 1}
\]
such that $m_1(a_n) = (-1)^n$ for $n = 1, \dotsc, 2\ell - 1$ and
$m_1 \le h$ in $(-\infty, b_1] \cup [b_2, b_3] \cup \dotsb \cup [b_{2\ell - 2}, \infty)$
and $m_1 \ge h$ in $[b_1, b_2] \cup [b_3, b_4] \cup \dotsb \cup [b_{2\ell - 3}, b_{2\ell - 2}]$.
\end{theorem}

This means that the picture on the right of Figure \ref{fig:h<1} is qualitatively accurate. The result is
also consistent with the idea that we should think of these minimisers as a composition of several
N\'eel walls in a row.

It is an open question whether the minimisers of $E_h$ in $\A_h(\ell - \alpha/\pi)$ (in the 
cases where they exist) have a monotone phase. That is, if $m = (\cos \phi, \sin \phi)$ is
such a minimiser, does it follow that $\phi' \ge 0$ (or even $\phi' > 0$)?
The answer is known only for the simplest cases of a transition of degree $\pm \alpha/\pi$ or
$\pm (1 - \alpha/\pi)$, where a standard symmetrisation argument applies (see \cite{Chermisi-Muratov:13, Me1}).
For a higher degree $\ell - \alpha/\pi$ with $\ell\ge 2$, Theorem \ref{thm:structure} is consistent
with a monotone phase, but it does of course not answer the question.

It is known, however, that the solutions of our minimisation problem are symmetric up to translation in the
following sense \cite[Lemma 3.2]{Ignat-Moser:17}: if $d \in \N - \alpha/\pi$ and $m \in \A_h(d)$ is a
minimiser of $E_h$ in $\A_h(d)$, then there exists $t_0 \in \R$ such that
\[
m_1(t_0 - x_1) = m_1(t_0 + x_1) \quad \text{and} \quad m_2(t_0 - x_1) = - m_2(t_0 + x_1)
\]
for all $x_1 \in \R$.

Another open question is whether minimisers of $E_h$ in $\A_h(\ell-\alpha/\pi)$ are unique (up to translation in $x_1$) for $\ell\geq 2$. The answer is yes for $\A_h(\alpha/\pi)$ and for $\A_h(1-\alpha/\pi)$,
as the energy is strictly convex in the $m_1$-component\footnote{Minimising $E_h$ in $\A_h(\alpha/\pi)$ or in $\A_h(1-\alpha/\pi)$ is in fact equivalent to minimising $E_h$ under the constraint $m_1(0) = 1$ or $m_1(0) = -1$, respectively.} (see \cite[Proposition 1]{Ignat-Moser:16}).

\subsection{Scaling}

The three terms in the energy $E_h(m)$ have different scaling. However, 
after rescaling in the variable $x_1$ and renormalizing the energy, only one length scale remains.
This is why in the physical model, there is, in general, a parameter $\eps$ in front of the exchange energy. As our results are qualitative (and not necessarily quantitative), we fix that parameter $\eps=1$. The critical values
$H_\ell$, $H_\ell^0$, $H_\ell^+$, and $H_\ell^-$ in our main results must of course be expected to
depend on $\epsilon$. However, as we do not attempt to give the optimal values,
we do not discuss this question any further.

For a critical point $m$ of the functional $E_h$, the Pohozaev identity (see \cite[Proposition~1.1]{Ignat-Moser:17}) implies that we have equipartition of the energy coming from the local terms (the exchange and anisotropy energies).
This equality effectively fixes the length scale of the core of a N\'eel wall (which is of order $l \sim1$ for $\eps=1$). The nonlocal term is dominant and has the length scale of the tails much larger than $l$.

\subsection{Relation to other models}

As we have mentioned in the introduction, the model studied in this paper can be seen as a nonlocal ``perturbation" of the (local) Allen-Cahn model (the latter consisting only in the exchange and anisotropy terms), see \cite[Appendix]{Ignat-Moser:17}. The nonlocal term in $E_h$ is in fact the key ingredient for existence of transition layers with higher winding number.

We can also relate our model to the study of $\frac12$-harmonic maps defined on the real axis with values into the unit sphere (see, e.g., \cite{Lio_Riv, Lio} for regularity, compactness and bubble analysis and \cite{Millot-Sire:15}
for a Ginzburg-Landau approximation). There is a model for boundary vortices in micromagnetics (see e.g. \cite{KohnSlastikov-2005, Kurzke-2006, Kurzke-03, Moser_ARMA, Moser_CPAM}) that amounts to a ``perturbation" of the $\frac12$-harmonic map problem by a zero-order term comparable to the anisotropy in our model.
But the problem
studied in this paper has an additional higher-order term, which changes the behaviour of the
problem dramatically. (This is most striking in the asymptotic analysis carried out in
an earlier paper \cite{Ignat-Moser:16}, where the interaction between different N\'eel walls is
studied. We have attraction where the model for boundary vortices would give repulsion and vice versa.)

\paragraph{Acknowledgment.} R.I. acknowledges partial support by the ANR project ANR-14-CE25-0009-01.

\section{Preliminaries} \label{sect:prelim}

In this section we discuss a few tools for the analysis of our problem and
recall some known results. Since the nonlocal term in the energy functional (i.e., the stray field energy)
is not only the most challenging to analyse, but in fact determines the
behaviour of the system to a considerable extent, it will have a prominent place here.

\subsection{Representations of the stray field energy}

We have already seen two different representations of the stray field energy.
One of them is given by \eqref{eqn:double_integral} and will be used in some of our
estimates later on. The other representation involves the stray field potential $u$
that is determined by the boundary value problem \eqref{eqn:u_harmonic_higher}, \eqref{eqn:boundary_condition_higher}.
In order to make the discussion of the problem rigorous, we introduce the inner product
$\scp{\blank}{\blank}_{\dot{H}^1(\R_+^2)}$ on the set of all $\phi \in C^\infty(\R \times [0, \infty))$
with compact support in $\R \times [0, \infty)$ (so $\phi$ is allowed to take non-zero values on the
boundary $\R \times \{0\}$). This inner product is defined by the formula
\[
\scp{\phi}{\psi}_{\dot{H}^1(\R_+^2)} = \int_{\R_+^2} \nabla \phi \cdot \nabla \psi \, dx.
\]
The space $\dot{H}^1(\R_+^2)$ is then the completion of the resulting inner product space.
Its elements are not functions, strictly speaking, as the completion will
conflate all constants. Nevertheless, we will sometimes implicitly pick a specific constant
(for example, by considering the limit at $\infty$) and treat elements of
$\dot{H}^1(\R_+^2)$ as functions.

Given $m_1 \in h + H^1(\R)$, there exists
a unique solution $u \in \dot{H}^1(\R_+^2)$ of the boundary value problem
\eqref{eqn:u_harmonic_higher}, \eqref{eqn:boundary_condition_higher}.
This solution will satisfy
\[
\|m_1 - h\|_{\dot{H}^{1/2}(\R)}^2 = \int_{\R_+^2} |\nabla u|^2 \, dx.
\]

While it is sometimes convenient to work with $u$, there is also a dual problem
that is more useful for other purposes. Namely, if $u \in \dot{H}^1(\R_+^2)$ solves
\eqref{eqn:u_harmonic_higher}, \eqref{eqn:boundary_condition_higher}, then we consider
$\nabla^\perp u = (-\dd{u}{x_2}, \dd{u}{x_1})$, which will satisfy $\curl \nabla^\perp u = 0$
in $\R_+^2$. By the Poincar\'e lemma, there exists $v \colon \R_+^2 \to \R$ such that $\nabla v = \nabla^\perp u$
(i.e., such that $u$ and $v$ are conjugate harmonic functions).
This implies that $\Delta v = 0$ in $\R_+^2$ and $\dd{v}{x_1}(\blank, 0) = m_1'$ in $\R$.
After adding a suitable constant, we thus obtain a solution of the boundary value problem
\begin{alignat}{2}
\Delta v & = 0 & \quad & \text{in $\R_+^2$}, \label{eqn:v_harmonic} \\
v & = m_1 - h && \text{on $\R \times \{0\}$}. \label{eqn:dual_boundary_condition}
\end{alignat}
If we are content to fix $v$ only up to a constant, then we may regard it as
an element of $\dot{H}^1(\R_+^2)$. Of course we also have the identity
\[
\|m_1 - h\|_{\dot{H}^{1/2}(\R)}^2 = \int_{\R_+^2} |\nabla v|^2 \, dx,
\]
which may be more familiar to the reader as $v$ is the harmonic extension of $m_1 - h$ to
the upper half-plane.

\subsection{The Euler-Lagrange equation}

As we are interested in minimising the functional $E_h$ in $\A_h(d)$, we will study
the Euler-Lagrange equation for critical points of $E_h$. Given $m \colon \R \to \Ss^1$,
it is convenient to represent the Euler-Lagrange equation in terms of the lifting $\phi \colon \R \to \R$.
That is, we write $m = (\cos \phi, \sin \phi)$, and then the equation becomes
\begin{equation} \label{eqn:Euler-Lagrange_higher}
\phi'' = (h - \cos \phi + u'(\blank, 0)) \sin \phi \quad \text{in $\R$}.
\end{equation}
Here $u \in \dot{H}^1(\R_+^2)$ is the stray field potential as introduced in the preceding
section and we use the abbreviation $u' = \dd{u}{x_1}$. The derivation of this equation
is almost identical to the corresponding calculations given in our previous work
\cite{Ignat-Moser:16}. It is known \cite[Proposition~3.1]{Ignat-Moser:17}
that solutions of \eqref{eqn:Euler-Lagrange_higher} must be smooth.

If $m \in \A_h(d)$ for a given winding number $d$, then there will be at least a certain
number of points, say $a_1, \dotsc, a_N \in \R$, where $m_1(a_n) = \pm 1$ and $m_2(a_n) = 0$
(i.e., $\phi(a_n) \in \pi \Z$). We can use these points as a proxy for the positions of the
N\'eel walls in the given configuration. One of the tasks for the proofs of the main theorems
will be to estimate the rate of decay of $m_1 - h$ as we move away from one of the points $a_1, \dotsc, a_N$.
But some information about these points is already available from our previous work
\cite[Lemma 3.1]{Ignat-Moser:17}, namely that for energy minimising solutions of the Euler-Lagrange equation,
the number $N$ is determined uniquely by the prescribed winding number.
(Although it is assumed that $h \not= 1$ in the other paper, the proof of this
statement does not depend on the assumption.)

\begin{lemma}\label{lem:no_unaccounted_walls}
For any $h \in [0, 1]$ and $d \in \Z \pm \{0, \alpha/\pi\}$, the following holds true.
\begin{enumerate}
\item \label{item:no_unaccounted_walls}
Suppose that $d \not= 0$ and $m \in \A_h(d)$ minimises $E_h$ in $\A_h(d)$. Then
\[\hspace{-0.5cm}
|m_1^{-1}(\{\pm 1\})| = \begin{cases}
2|d| - 1 & \text{if $h = 1$ and $d \in \Z$},\\
2 |d| & \text{if $h < 1$ and $d \in \Z$}, \\
2 \ell - 1 & \text{if $h < 1$ and $|d| = \ell-1+ \alpha/\pi$ or $|d| = \ell- \alpha/\pi$ with $\ell \in \N$}.
\end{cases}
\]
Furthermore, if $a \in \R$ with $m_1(a) = \pm 1$, then
$m_2'(a) \not= 0$. (Therefore, any lifting $\phi \colon \R \to \R$ of $m$ will satisfy $\phi'(a)\neq 0$.)
\item \label{item:lower_energy_bound}
If $0 < |d| < 1$, then $\E_h(d) \geq (1 - \arccos(\pi d))^2$.
If $|d| \ge 1$, then $\mathcal{E}_h(d) \ge 2|d| - 1$. In particular, $\E_h(d)>0$ for every degree $d\neq 0$.
\end{enumerate}
\end{lemma}

\begin{proof}
Statement \ref{item:no_unaccounted_walls} was proved in \cite[Lemma 3.1]{Ignat-Moser:17}.

For statement \ref{item:lower_energy_bound}, we observe the following.
If $h \not= 0$ and $m \in \A_h(\pm \alpha/\pi)$, then there exists $a \in \R$ such that $m_1(a) = 1$, while
$\lim_{x_1 \to \pm\infty} m_1(x_1) = h$. Thus
\[
\int_a^\infty (|m'|^2 + (m_1 - h)^2) \, dx_1 \geq 2 \int_a^\infty |m_1'| |m_1 - h|  \, dx_1 \geq (1 - h)^2.
\]
A similar estimate holds for the integral over $(-\infty, a)$.
Hence $E_h(m) \ge (1 - h)^2$.

If $h = 0$ and $m \in \A_h(\pm 1/2)$, then we may have $a \in \R$ with $m_1(a) = -1$ instead, but this
situation permits the same arguments.

If $m\in \A_h(d)$ for $d \not\in \{0, \pm \alpha/\pi\}$, then there exist at least $N \ge 2|d| - 1$ points, say
$a_1, \dotsc, a_N \in \R$, with $a_1 < \dotsb < a_N$ and $m_1(a_n) = (-1)^n$ for $n = 1, \dotsc, N$.
Furthermore, there exist $b_1, \dotsc, b_{N - 1} \in \R$ such that $a_n < b_n \le a_{n + 1}$ and $m_1(b_n) = h$ for
$n = 1, \dotsc, N - 1$. Set $b_0 = -\infty$ and $b_N = \infty$. Then by the same arguments as before,
\[
\int_{a_n}^{b_n} (|m'|^2 + (m_1 - h)^2) \, dx_1 \geq \left((-1)^n - h\right)^2, \quad n = 1, \dotsc, N.
\]
Considering the intervals $(b_{n - 1}, a_n)$ as well and summing over $n$, we then find that
$E_h(m) \ge (1 + h)^2$ if $d = \pm(1 - \alpha/\pi)$
and $E_h(m) \ge 2|d| - 1$ otherwise.
\end{proof}

\begin{remark}
\label{rem:degree}
We recall that the following was proved in  \cite[Propositions 2.2 and 2.3]{Ignat-Moser:17}.
\begin{enumerate}
\item (Monotonicity) If $d_1, d_2\in \Z+\{0,\pm \alpha /\pi\}$ with $0\leq d_1\leq d_2$ (and if $0<h<1$, we suppose that $(d_1, d_2)\neq (\ell+\alpha/\pi, 1+\ell-\alpha/\pi)$ for $\ell\in \Z$), then $\E_h(d_1) \leq \E_h(d_2)$.
This is because a transition of degree $d_2$ contains also a (sub)transition of degree $d_1$,
except for the exceptional case described above.

\item (Subadditivity) If $d_1, d_2, d \in \Z+\{0,\pm \alpha /\pi\}$ with $d_1+d_2=d$ (if $h=\cos \frac\pi 3$ and $d_2-d_1\in \Z$, we suppose that $d\in \Z$), then $\E_h(d) \leq \E_h(d_1)+ \E_h(d_2)$.
This is because the concatenation of two transitions of degrees $d_1$ and $d_2$ has more energy than an optimal transition of degree  $d_1+d_2$. Two such neighbouring transitions are compatible if 
either $d_1$ or $d_2$ is an integer, or if $d_1+d_2\in \Z$; this explains the constraint above for
$h=\cos \frac\pi 3$.
\end{enumerate}
\end{remark}

\subsection{$H^1$ and $H^2$-estimates away from the N\'eel walls}

The following is a consequence of the Euler-Lagrange equation, obtained
by the use of suitable test functions and differentiation. It provides local estimates
for a solution of \eqref{eqn:Euler-Lagrange_higher} away from
the points $a_1, \dotsc, a_N$ where $m_1$ takes the value $\pm 1$
(thought of as the locations of the N\'eel walls and discussed in the preceding section).

The following inequalities include the energy density of the stray field energy, given in terms of
the solution $v \in \dot{H}^1(\R_+^2)$ of \eqref{eqn:v_harmonic}, \eqref{eqn:dual_boundary_condition}.
As a consequence, we have inequalities involving integrals over $\R$ and over $\R_+^2$ simultaneously.
For convenience, we use the following shorthand notation: given a function $f \colon \R_+^2 \to \R$,
we write
\[
\int_{-\infty}^\infty f \, dx_1 = \int_{-\infty}^\infty f(x_1, 0) \, dx_1.
\]

\begin{proposition} \label{prop:local_energy_estimates}
Let $h \in [0, 1]$. Suppose that $I \subseteq \R$ is an open set
and $\phi \in H_\loc^1(\R)$. Let $m = (\cos \phi, \sin \phi)$ and suppose that $m_1 - h \in \dot{H}^{1/2}(\R)$ and
$m_2(x_1) \not= 0$ for all $x_1 \in I$.
Let $\eta \in C_0^\infty(\R^2)$ with $\supp \eta(\blank, 0) \subseteq I$.
Suppose that $v \in \dot{H}^1(\R_+^2)$ is the unique solution of $\Delta v = 0$ in $\R_+^2$
with $v(x_1, 0) = m_1(x_1) - h$ for all $x_1 \in \R$. If $\phi$
is a solution of \eqref{eqn:Euler-Lagrange_higher} in $I$, then
\begin{multline} \label{eqn:first_order_inequality}
\int_{-\infty}^\infty \eta^4 \left(\frac{1}{2} |m'|^2 + (m_1 - h)^2\right) \, dx_1 + \int_{\R_+^2} \eta^4 |\nabla v|^2 \, dx \\
\le 576 \int_{-\infty}^\infty (\eta')^4 \, dx_1 + 16 \int_{\R_+^2} v^2 \eta^2 |\nabla \eta|^2 \, dx
\end{multline}
and
\begin{multline} \label{eqn:second_order_inequality}
\int_{-\infty}^\infty \eta^2 \left(|m''|^2 + (m_1')^2 + \frac{|m'|^4}{m_2^2}\right) \, dx_1 + \int_{\R_+^2} \eta^2 |\nabla^2 v|^2 \, dx \\
\le 32 \int_{-\infty}^\infty (\eta')^2 |m'|^2 \, dx_1 + 24 \int_{\R_+^2} |\nabla \eta|^2 |\nabla v|^2 \, dx.
\end{multline}
\end{proposition}

\begin{proof}
If $u \in \dot{H}^1(\R_+^2)$ is the solution of \eqref{eqn:u_harmonic_higher},
\eqref{eqn:boundary_condition_higher}, then $u$ and $v$ are conjugate harmonic functions
in the sense that $\nabla v = \nabla^\perp u$. Thus
equation \eqref{eqn:Euler-Lagrange_higher} may be written in the form
\[
\phi'' = \left(h - \cos \phi + \dd{v}{x_2}\right) \sin \phi.
\]
Since $\sin \phi\neq 0$ in $I$, then
\[
v(\blank, 0) \dd{v}{x_2}(\blank, 0) = \phi'' \frac{\cos \phi - h}{\sin \phi} + (\cos \phi - h)^2
\]
in $I$. It follows that
\[
\begin{split}
\int_{\R_+^2} \eta^4 |\nabla v|^2 \, dx & = - \int_{-\infty}^\infty \eta^4 v \dd{v}{x_2} \, dx_1 - 4 \int_{\R_+^2} \eta^3 v \nabla \eta \cdot \nabla v \, dx \\
& = - \int_{-\infty}^\infty \eta^4 (\cos \phi - h)^2 \, dx_1 + \int_{-\infty}^\infty \eta^4 (\phi')^2 \frac{h \cos \phi - 1}{\sin^2 \phi} \, dx_1 \\
& \quad + 4 \int_{-\infty}^\infty \eta^3 \eta' \phi' \frac{\cos \phi - h}{\sin \phi} \, dx_1 - 4\int_{\R_+^2} \eta^3 v \nabla \eta \cdot \nabla v \, dx.
\end{split}
\]
We claim that
\begin{equation} \label{eqn:inequality_for_cos_phi}
|\cos \phi - h| \le 3(1 - h \cos \phi).
\end{equation}
For $0\leq h \le \frac{1}{3}$, this is clear as $|\cos \phi - h| \le 2$ and $1 - h \cos \phi \ge \frac{2}{3}$
in this case. For $1\geq h > \frac{1}{3}$, we note that on the one hand,
\[
h (\cos \phi - h) \le 1 - h\cos \phi,
\]
because
\[
2 \cos \phi\leq 2 \le \frac{1}{h} + h,
\]
and on the other hand,
\[
h (h - \cos \phi) \le 1 - h\cos \phi
\]
trivially. Hence \eqref{eqn:inequality_for_cos_phi} follows.

In particular, using Young's inequality with three factors and exponents $2$, $4$, and $4$, we may estimate
\[
\begin{split}
4 \int_{-\infty}^\infty \eta^3 \eta' \phi' \frac{\cos \phi - h}{\sin \phi} \, dx_1 & \le \frac{1}{2} \int_{-\infty}^\infty \eta^4 (\phi')^2 \frac{|\cos \phi - h|}{3\sin^2 \phi} \, dx_1 \\
& \quad + \frac{1}{2} \int_{-\infty}^\infty \eta^4 (\cos \phi - h)^2 \, dx_1 + 288 \int_{-\infty}^\infty (\eta')^4 \, dx_1 \\
&\stackrel{\eqref{eqn:inequality_for_cos_phi}}{\le} \frac{1}{2} \int_{-\infty}^\infty \eta^4 (\phi')^2 \frac{1 - h \cos \phi}{\sin^2 \phi} \, dx_1 \\
& \quad + \frac{1}{2} \int_{-\infty}^\infty \eta^4 (\cos \phi - h)^2 \, dx_1 + 288 \int_{-\infty}^\infty (\eta')^4 \, dx_1.
\end{split}
\]
Furthermore,
\[
-4 \int_{\R_+^2} \eta^3 v \nabla \eta \cdot \nabla v \, dx \le \frac{1}{2} \int_{\R_+^2} \eta^4 |\nabla v|^2 \, dx + 8 \int_{\R_+^2} v^2 \eta^2 |\nabla \eta|^2 \, dx.
\]
Hence
\begin{multline*}
\frac{1}{2} \int_{-\infty}^\infty \eta^4 \left((\phi')^2 \frac{1 - h \cos \phi}{\sin^2 \phi} + (\cos \phi - h)^2\right) \, dx_1 + \frac{1}{2} \int_{\R_+^2} \eta^4 |\nabla v|^2 \, dx \\
\le 288 \int_{-\infty}^\infty (\eta')^4 \, dx_1 + 8 \int_{\R_+^2} v^2 \eta^2 |\nabla \eta|^2 \, dx.
\end{multline*}
By Lemma \ref{lem:auxiliary_inequality} in the appendix,
\[
\frac{1 - h \cos \phi}{\sin^2 \phi} \geq \frac12.
\]
Thus, we obtain the first
inequality.

The second inequality is a direct consequence of an estimate proved
in our paper \cite{Ignat-Moser:17}. Let $u \in \dot{H}^1(\R_+^2)$ be the solution of
\eqref{eqn:u_harmonic_higher}, \eqref{eqn:boundary_condition_higher}.
In the proof of \cite[Lemma 3.3]{Ignat-Moser:17}, it is shown
that under the above assumptions,
\begin{multline*}
\int_{-\infty}^\infty \eta^2 \left((\phi'')^2 + 2(\phi')^2 \sin^2 \phi + \frac{1}{3}(\phi')^4 (1 + \cot^2 \phi)\right) \, dx_1 + \frac{1}{2} \int_{\R_+^2} \eta^2 |\nabla^2 u|^2 \, dx \\
\le \frac{16}{3} \int_{-\infty}^\infty (\eta')^2 (\phi')^2 \, dx_1 + 4 \int_{\R_+^2} |\nabla \eta|^2 |\nabla u|^2 \, dx.
\end{multline*}
Noting that $|\nabla u| = |\nabla v|$ and $|\nabla^2 u| = |\nabla^2 v|$, and that
\[
|m'|^2 = (\phi')^2, \quad (m_1')^2 = (\phi')^2 \sin^2 \phi, \quad \text{and} \quad |m''|^2 = (\phi'')^2 + (\phi')^4,
\]
we derive inequality \eqref{eqn:second_order_inequality}.
\end{proof}

\section{Decay estimates}

We now study how fast $m_1 - h$ decays when we move away from the points where
$m_1$ takes one of the values $\pm 1$. For technical reasons, we proceed in two
steps, first proving a preliminary estimate before improving it in the second step.

\subsection{A preliminary $L^\infty$-estimate}

\begin{lemma} \label{lem:preliminary_decay_estimate}
There exists a constant $C > 0$ such that the following inequality holds true.
Suppose that $h \in [0, 1]$ and
$m \in H_\loc^1(\R; \Ss^1)$ is a critical point of $E_h$. For any $t \in \R$,
let
\[
\sigma(t) = 1 + \inf \set{|t - x_1|}{x_1 \in \R \text{ with } m_2(x_1) = 0}.
\]
Then
\[
|m_1(t) - h| \le C(\sigma(t))^{-3/4} \sqrt{E_h(m) + 1}
\]
for all $t \in \R$.
\end{lemma}

\begin{proof}
Let $v$ denote the solution of \eqref{eqn:v_harmonic}, \eqref{eqn:dual_boundary_condition}.

As $|m_1 - h|\leq 2$, it suffices to consider $t \in \R$ with $\sigma(t) \ge 4$. Let $R = \frac{\sigma(t)}{4}$ (so $R\geq 1$).
Choosing a suitable cut-off function $\eta$ in \eqref{eqn:first_order_inequality} in Proposition \ref{prop:local_energy_estimates},
we see that
\[
\int_{t - R}^{t + R} \left(|m'|^2 + (m_1 - h)^2\right) \, dx_1 \le \frac{C_1 \pi^2}{32R^2} \int_{\R \times (0, R)} v^2 \, dx + \frac{C_1}{R^3}
\]
for a universal constant $C_1$.

The boundary value problem \eqref{eqn:v_harmonic}, \eqref{eqn:dual_boundary_condition} gives rise
to some estimates for $v$ with standard tools. In particular,
according to Lemma \ref{lem:L^2-estimate_for_v} in the appendix (applied for $p = 2$),
\[
\begin{split}
\int_{\R \times (0, R)} v^2 \, dx \le \frac{16 R}{\pi^2} \int_{-\infty}^\infty (m_1 - h)^2 \, dx_1.
\end{split}
\]
This, combined with the above inequality, yields
\begin{equation} \label{eqn:energy_estimate_away_from_walls}
\int_{t - R}^{t + R} \left(|m'|^2 + (m_1 - h)^2\right) \, dx_1 \le \frac{C_1}{R}(E_h(m) + 1).
\end{equation}
Moreover, estimate \eqref{eqn:second_order_inequality} in Proposition \ref{prop:local_energy_estimates} leads to 
\[
\int_{t - R}^{t + R} (m_1')^2 \, dx_1 \le \frac{C_2}{2R^2} E_h(m)
\]
for a universal constant $C_2$. In particular, we deduce that
\be
\label{oscil}
\left(\osc_{[t - \sqrt{R}, t + \sqrt{R}]} m_1\right)^2 \le \left(\int_{t - \sqrt{R}}^{t + \sqrt{R}} |m_1'| \, dx_1\right)^2 \le \frac{C_2}{R^{3/2}} E_h(m).
\ee

Set
\[
C_0 = \sqrt{\frac{C_1}{2}} + \sqrt{C_2}.
\]
We claim that
\[
|m_1(t) - h| \le \frac{C_0}{R^{3/4}} \sqrt{E_h(m) + 1},
\]
which will conclude the proof. Indeed, if we had the inequality
\[
|m_1(t) - h| > \frac{C_0}{R^{3/4}} \sqrt{E_h(m) + 1},
\]
then it would follow from \eqref{oscil} that
\[
|m_1(x_1) - h| > \frac{C_0 - \sqrt{C_2}}{R^{3/4}} \sqrt{E_h(m) + 1} = \sqrt{\frac{C_1 (E_h(m) + 1)}{2R^{3/2}}}
\]
for all $x_1 \in [t - \sqrt{R}, t + \sqrt{R}]$. Hence
\[
\int_{t - \sqrt{R}}^{t + \sqrt{R}} (m_1 - h)^2 \, dx_1 > \frac{C_1}{R}(E_h(m) + 1),
\]
which contradicts \eqref{eqn:energy_estimate_away_from_walls}.
\end{proof}

\subsection{A preliminary $L^1$-estimate}

Eventually we want to estimate the $L^1$-norms of the positive and negative parts of $m_1 - h$.
The preceding inequality is not good enough for this purpose, but we will use it as a first step.
In the next step (Lemma \ref{lem:decay_estimate}), we derive an $L^2$-estimate of $m_1 - h$, before turning it into an $L^1$-estimate in Proposition \ref{prop:L^1-estimate} below.

\begin{lemma} \label{lem:decay_estimate}
Let $p \in (\frac{4}{3}, 2]$. Then there exists a number $C > 0$ with the following property.
Let $h \in [0, 1]$ and $d \in \Z \pm \{0, \alpha/\pi\}$. Suppose that $m \in \A_h(d)$
is a minimiser of $E_h$ in $\A_h(d)$. Let $a_1, a_2 \in \R \cup \{\pm \infty\}$ with
$a_1 < a_2$ such that $m_2 \not= 0$ in $(a_1, a_2)$. Then
\[
\int_{a_1 + R}^{a_2 - R} \left(|m'|^2 + (m_1 - h)^2\right) \, dx_1 \le C (|d| + 1)^{2/p}R^{-2/p} (E_h(m) + 1)
\]
for all $R \ge 1$.
\end{lemma}

\begin{proof}
As in the proof of Proposition \ref{prop:local_energy_estimates} and Lemma \ref{lem:preliminary_decay_estimate},
let $v$ denote the harmonic extension of $m_1 - h$ to $\R_+^2$.
Then by Lemma \ref{lem:L^2-estimate_for_v} in the appendix, there exists a constant $C_1 = C_1(p)$ such that
\[
\int_0^R \int_{-\infty}^\infty v^2 \, dx_1 \, dx_2 \leq C_1 R^{2 - 2/p} \|m_1 - h\|_{L^p(\R)}^2.
\]

Assuming that we have a cut-off function $\eta \in C_0^\infty(\R^2)$ with $0 \le \eta \le 1$
such that $\supp \eta \subseteq \R \times (-\infty, R]$ and
$\supp \eta(\blank, 0) \subseteq (a_1, a_2)$, then \eqref{eqn:first_order_inequality} in 
Proposition \ref{prop:local_energy_estimates} implies that
\begin{multline*}
\int_{-\infty}^\infty \eta^4 \left(\frac{1}{2} |m'|^2 + (m_1 - h)^2\right) \, dx_1 + \int_{\R_+^2} \eta^4 |\nabla v|^2 \, dx \\
\le C_2 R^{2 - 2/p} \|\nabla \eta\|_{L^\infty(\R_+^2)}^2 \|m_1 - h\|_{L^p(\R)}^2
+ 576 \|\eta'\|_{L^4(\R)}^4,
\end{multline*}
where $C_2 = 16 C_1$.
A suitable choice of $\eta$ therefore gives rise to a number $C_3$, depending only
on $p$, such that
\[
\int_{a_1 + R}^{a_2 - R} \left(|m'|^2 + (m_1 - h)^2\right) \, dx_1 \le C_3 R^{-2/p} \|m_1 - h\|_{L^p(\R)}^2 + C_3 R^{-3}.
\]
Furthermore, since $p>\frac43$, the function $\sigma$ in Lemma \ref{lem:preliminary_decay_estimate} has
the property that $\sigma^{-3p/4}$ is integrable
over $\R$. Lemma \ref{lem:no_unaccounted_walls} and Lemma \ref{lem:preliminary_decay_estimate} then imply that
\[
\|m_1 - h\|_{L^p(\R)} \le C_4 (|d| + 1)^{1/p} \sqrt{E_h(m) + 1}
\]
for some constant $C_4 = C_4(p)$. Thus we obtain a constant $C_5 = C_5(p)$ such that
\[
\int_{a_1 + R}^{a_2 - R} \left(|m'|^2 + (m_1 - h)^2\right) \, dx_1 \le C_5 (|d| + 1)^{2/p} R^{-2/p} (E_h(m) + 1).
\]
This is the desired estimate.
\end{proof}

The preceding result implies an $L^1$-estimate for
the positive and negative parts of $m_1 - h$. In the following, we write $(m_1 - h)_+ = \max\{m_1 - h, 0\}\geq 0$
and $(m_1 - h)_- = \min\{m_1 - h, 0\}\leq 0$.

\begin{proposition} \label{prop:L^1-estimate}
Let $p \in (\frac{4}{3}, 2)$. Then there exists a number $C > 0$ with the following property.
Let $h \in [0, 1]$. Given $d \in \Z \pm \{0, \alpha/\pi\}$,
suppose that $m \in \A_h(d)$ is a minimiser of $E_h$ in $\A_h(d)$. Then
\[
\int_{-\infty}^\infty (m_1 - h)_+ \, dx_1 \le C (|d| + 1)^{1 + 1/p} \alpha^{2(2 - p)/(2 + p)} \sqrt{E_h(m) + 1}
\]
and
\[
-\int_{-\infty}^\infty (m_1 - h)_- \, dx_1 \le C (|d| + 1)^{1 + 1/p} \sqrt{E_h(m) + 1}.
\]
\end{proposition}

\begin{proof}
Let $C_1$ be the constant satisfying the statement of Lemma \ref{lem:decay_estimate}
for the given value of $p$ and define
\[
c_0 = \frac{1}{\sqrt{C_1 (|d| + 1)^{2/p} (E_h(m) + 1)}}.
\]
By Lemma \ref{lem:no_unaccounted_walls}, there exist $a_1, \dotsc, a_N \in \R$ with
$a_1 < \dotsb < a_N$ and $N \le 2|d| + 1$ such that $m_2(a_n) = 0$ for $n = 1, \dotsc, N$
and $m_2 \not= 0$ in $\R \setminus \{a_1, \dotsc, a_N\}$. Set $a_0=-\infty$ and $a_{N+1}=+\infty$. Fix $n \in \{1, \dotsc, N\}$.

Lemma \ref{lem:decay_estimate} gives $L^2$-estimates for $m_1 - h$ at first,
but using it for varying values of $R$, we can derive an $L^1$-estimate as well.
The details are given in Lemma \ref{lem:decay_implies_L^1_estimate} in the appendix.
We apply this Lemma \ref{lem:decay_implies_L^1_estimate} to the functions
\[
\psi_1(x_1) = \begin{cases}
c_0 |m_1(a_n + x_1) - h| & \text{if $1 \le x_1 \le (a_{n + 1} - a_n)/2$}, \\
0 & \text{if $x_1 > (a_{n + 1} - a_n)/2$},
\end{cases}
\]
and
\[
\psi_2(x_1) = \begin{cases}
c_0 |m_1(a_n - x_1) - h| & \text{if $1 \le x_1 \le (a_n - a_{n - 1})/2$}, \\
0 & \text{if $x_1 > (a_n - a_{n - 1})/2$},
\end{cases}
\]
and use Lemma \ref{lem:decay_estimate} to verify that the hypothesis of Lemma \ref{lem:decay_implies_L^1_estimate}
is satisfied for $\sigma = 2/p$. (This is why we assume $p<2$ here, in contrast to Lemma \ref{lem:decay_estimate}.)
Therefore, there exists a constant $C_2 = C_2(p)$
such that
\[
\int_{(a_n + a_{n - 1})/2}^{a_n - R} |m_1 - h| \, dx_1 \le \frac{C_2R^{1/2 - 1/p}}{c_0}
\]
and
\[
\int_{a_n + R}^{(a_{n + 1} + a_n)/2} |m_1 - h| \, dx_1 \le \frac{C_2R^{1/2 - 1/p}}{c_0}
\]
for all $R \ge 1$.

If $\alpha \le 1$, then we choose $R = \alpha^{-4p/(2 + p)}$. Using the fact that  
$(m_1 - h)_+ \le 1 - h \le \frac{\alpha^2}{2}$
in $(a_n-R, a_n+R)$, we then obtain the estimate
\[
\int_{(a_n + a_{n - 1})/2}^{(a_{n + 1} + a_n)/2} (m_1 - h)_+ \, dx_1 \le \left(\frac{2C_2}{c_0} + 1\right) \alpha^{2(2 - p)/(2 + p)}.
\]
Then it suffices to sum over $n = 1, \dotsc, N$ to prove the first inequality.

For $\alpha > 1$ and for the second inequality, we can use the same arguments with $R = 1$, since 
$|m_1 - h|\leq 2$ in $(a_n-R, a_n+R)$.
\end{proof}

\subsection{Improved estimates}

With this control of the $L^1$-norm, we can now take advantage of the second
inequality in Lemma \ref{lem:L^2-estimate_for_v} and improve the estimates again.

\begin{proposition} \label{prop:improved_decay}
There exists a number $C > 0$ with the following property.
Let $h \in [0, 1]$ and suppose that $m \in \A_h(d)$
is a minimiser of $E_h$ in $\A_h(d)$ for some $d \in \Z \pm \{0, \alpha/\pi\}$. Let $u \in \dot{H}^1(\R_+^2)$ denote the
solution of \eqref{eqn:u_harmonic_higher}, \eqref{eqn:boundary_condition_higher},
and let $a_1, a_2 \in \R \cup \{\pm \infty\}$ with $a_1 < a_2$ and $m_2 \not= 0$ in $(a_1, a_2)$.  Then
\begin{multline} \label{eqn:improved_first_order_inequality}
\int_{a_1 + R}^{a_2 - R} \left(|m'|^2 + (m_1 - h)^2\right) \, dx_1 + \int_{a_1 + R}^{a_2 - R} \int_0^R |\nabla u|^2 \, dx_2 \, dx_1 \\
\le C(|d| + 1)^{16/5}R^{-2} \log R \, (E_h(m) + 1)
\end{multline}
and
\begin{multline} \label{eqn:improved_second_order_inequality}
\int_{a_1 + R}^{a_2 - R} \left(|m''|^2 + (m_1')^2 + \frac{|m'|^4}{m_2^2}\right) \, dx_1 + \int_{a_1 + R}^{a_2 - R} \int_0^R |\nabla^2 u|^2 \, dx_2 \, dx_1 \\
\le C(|d|+1)^{16/5}R^{-4} \log R \, (E_h(m) + 1)
\end{multline}
and
\begin{equation} \label{eqn:improved_inequality_for_u'}
\int_{a_1 + R}^{a_2 - R} (u'(x_1, 0))^2 \, dx_1 \le C(|d| + 1)^{16/5} R^{-3} \log R \, (E_h(m) + 1)
\end{equation}
for all $R \ge 2$.
\end{proposition}

\begin{proof}
For the proof of \eqref{eqn:improved_first_order_inequality}, we simply combine \eqref{eqn:first_order_inequality} in Proposition \ref{prop:local_energy_estimates} with
the second and third inequality of Lemma \ref{lem:L^2-estimate_for_v} and Proposition \ref{prop:L^1-estimate} (with
$p = 5/3$).
Choosing a suitable cut-off function $\eta$ in Proposition \ref{prop:local_energy_estimates}, such that
$\eta \equiv 1$ in $[a_1 + R, a_2 - R] \times [0, R]$,
\begin{multline*}
\supp \nabla \eta \cap \R_+^2 \subseteq \bigl([a_1 + R/2, a_1 + R] \times (0, 2R]\bigr) \cup {} \\
\cup \bigl([a_2 - R, a_2 - R/2] \times (0, 2R]\bigr) \cup \bigl([a_1 + R/2, a_2 - R/2] \times [R, 2R]\bigr),
\end{multline*}
and $|\nabla \eta| \le 8/R$, we thus obtain \eqref{eqn:improved_first_order_inequality}.

In order to prove \eqref{eqn:improved_second_order_inequality}, we use \eqref{eqn:second_order_inequality} in Proposition \ref{prop:local_energy_estimates}
and choose $\eta$ similarly to
the first part of this proof again. Then we use \eqref{eqn:improved_first_order_inequality}
(say, with $R/2$ instead of $R$) to estimate the right-hand side of \eqref{eqn:second_order_inequality}.
This gives the desired estimate.

Finally, for the proof of \eqref{eqn:improved_inequality_for_u'}, we consider the
conjugate harmonic function $v \colon \R_+^2 \to \R$ with $\nabla v = \nabla^\perp u$.
Again we choose $\eta \in C_0^\infty(\R^2)$
and compute
\[
\begin{split}
\int_{-\infty}^\infty \eta^2 (u')^2 \, dx_1 & = \int_{-\infty}^\infty \eta^2 \left(\dd{v}{x_2}(x_1, 0)\right)^2 \, dx_1 \\
& = - \int_{\R_+^2} \div \left(\eta^2 \dd{v}{x_2} \nabla v\right) \, dx \\
& = - \int_{\R_+^2} \left(\eta^2 \nabla v \cdot \nabla \dd{v}{x_2} + 2\eta \dd{v}{x_2} \nabla \eta \cdot \nabla v\right) \, dx \\
& \le \int_{\R_+^2} \bigg(\eta^2 (R|\nabla^2 v|^2 + R^{-1} |\nabla v|^2) + 2|\eta| |\nabla \eta| |\nabla v|^2\bigg) \, dx.
\end{split}
\]
We observe that $|\nabla v| = |\nabla u|$ and $|\nabla^2 v| = |\nabla^2 u|$. Using
\eqref{eqn:improved_first_order_inequality} and \eqref{eqn:improved_second_order_inequality}
and choosing $\eta$ appropriately, we therefore obtain \eqref{eqn:improved_inequality_for_u'}.
\end{proof}

\section{Localisation} \label{sect:localisation}

In the proofs of the main results, we need to estimate how the energy changes when
we localise a given profile with a cut-off function. For $h < 1$, we have suitable
estimates from our previous work \cite[Proposition 2.1]{Ignat-Moser:17}. For
the case $h = 1$, however, we need to modify the result.

\begin{lemma} \label{lem:localisation}
There exists a constant $C>0$ with the following property. Suppose that $h=1$ and
$\phi \in H_\loc^{1}(\R)$ is such that
$m = (\cos \phi, \sin \phi)$ satisfies $E_1(m) < \infty$.
Let $v \in \dot{H}^1(\R_+^2)$ be the solution of \eqref{eqn:v_harmonic}, \eqref{eqn:dual_boundary_condition}.
Furthermore, suppose that there
exist two numbers $\ell_\pm \in 2\pi \Z$ and three measurable functions $\omega, \sigma, \tau \colon [0, \infty) \to [0, \infty)$
such that
\[
|\phi(x_1) - \ell_+| \le \omega(x_1) \quad \text{and} \quad |\phi(-x_1) - \ell_-| \le \omega(x_1) \quad \text{for all $x_1 \ge 0$}
\]
and
\[
|\phi'(x_1)| \le \sigma(|x_1|) \quad \text{and} \quad \left|\dd{v}{x_2}(x_1, 0)\right| \le \tau(|x_1|) \quad \text{for all $x_1 \in \R$.}
\]
Suppose that $\sup_{x_1 \ge r} \omega(x_1) \le \frac{\pi}{2}$ for some $r \ge 1$.
Then for any $R \ge r$ there exists $\tilde{m} \in H_\loc^{1}(\R; \Ss^1)$
such that
\begin{equation} \label{eqn:localised_degree}
\deg(\tilde{m}) = \frac{\ell_+ - \ell_-}{2\pi}
\end{equation}
and
\[
\tilde{m}_1 = 1 \text{ in $(-\infty, -2R] \cup [2R, \infty)$},  \quad \tilde m_1=m_1 \text{ in $[-R,R]$},
\]
and 
$|\tilde{m}_1 - 1| \le |m_1 - 1|$
everywhere, and such that
\begin{multline*}
E_1(\tilde{m}) \le E_1(m) + \frac{C}{R^2} \int_R^{2R} \left(\omega^2 + R\omega \sigma\right) \, dx_1 + C\int_R^\infty \omega^2\tau \, dx_1 \\
+ C \left(\int_R^\infty \omega^4 \, dx_1\right)^{1/2} \left( \int_R^\infty \left(\frac{\omega^4}{R^2} + \omega^2 \sigma^2 \right) \, dx_1\right)^{1/2}.
\end{multline*}
\end{lemma}

This result has a counterpart for $h \not= 1$, stated in \cite[Proposition 2.1]{Ignat-Moser:17}.
For the purpose of this paper, however, only the case $h < 1$ treated in \cite{Ignat-Moser:17} is relevant. The structure of the
statement is then similar to the above lemma, but for $h<1$ the inequality becomes
\begin{multline} \label{eqn:old_localisation}
E_h(\tilde{m}) \le E_h(m) + C \int_R^\infty \left(\frac{\omega^2}{R^2} + \sigma^2 + \omega\tau\right) \, dx_1 \\
+ C \left(\int_R^\infty \omega^2 \, dx_1\right)^{1/2} \left( \int_R^\infty \left(\frac{\omega^2}{R^2} + \sigma^2 \right) \, dx_1\right)^{1/2}.
\end{multline}
This is good enough for our proofs if $h < 1$, but for $h = 1$ we need the improvement given by 
Lemma \ref{lem:localisation}.

The following proof is similar to the reasoning of \cite[Proposition 2.1]{Ignat-Moser:17}, too,
but for the convenience of the reader, we repeat the arguments here.

\begin{proof}[Proof of Lemma \ref{lem:localisation}]
Choose $\eta \in C_0^\infty(\R)$ with $\eta(x_1) = 0$ for $|x_1| \ge 2R$ and $\eta(x_1) = 1$ for $|x_1| \le R$, and
such that $0 \le \eta \le 1$ and $|\eta'| \le 2/R$ everywhere. Define\footnote{An interpolation in $m_1$ was used in \cite[Proposition 2.1]{Ignat-Moser:17} for the case $h<1$ (and also for $h>1$). However, in our case $h=1$, the interpolation in the lifting $\phi$ is more appropriate.}
\[
\tilde{\phi}(x_1) = \begin{cases}
\ell_- + \eta(x_1) (\phi(x_1) - \ell_-) & \text{if $x_1 \le 0$}, \\
\ell_+ + \eta(x_1) (\phi(x_1) - \ell_+) & \text{if $x_1 > 0$}.
\end{cases}
\]
Set $\tilde{m} = (\cos \tilde{\phi}, \sin \tilde{\phi})$. Then clearly \eqref{eqn:localised_degree}
is satisfied and $|\tilde{m}_1 - 1| \le |m_1 - 1|$ everywhere.

For $x_1 > 0$, we compute
\[
\tilde{\phi}'(x_1) = \eta(x_1) \phi'(x_1) + \eta'(x_1) (\phi(x_1) - \ell_+).
\]
A similar identity holds for $x_1 < 0$. Hence
\[
(\tilde{\phi}')^2 \le (\phi')^2 + 2 |\eta \eta'| |\phi'| |\phi - \ell_\pm| + (\eta')^2 (\phi - \ell_\pm)^2 \le (\phi')^2 + \bigg(\frac{4\omega^2}{R^2} + \frac{4 \omega \sigma}{R}\bigg) {\bf 1}_{\{|x_1|\in [R, 2R]\}}
\]
(where for simplicity, we extend $\omega$ and $\sigma$ to $\R$ evenly).
It follows that
\[
\frac{1}{2} \int_{-\infty}^\infty (\tilde{\phi}')^2 \, dx_1 \le \frac{1}{2} \int_{-\infty}^\infty (\phi')^2 \, dx_1 + \frac{4}{R^2} \int_R^{2R} \left(\omega^2 + R \omega\sigma\right) \, dx_1.
\]
It is obvious that
\[
\frac{1}{2} \int_{-\infty}^\infty \left(\tilde{m}_1 - 1\right)^2 \, dx_1 \le \frac{1}{2} \int_{-\infty}^\infty (m_1 - 1)^2 \, dx_1.
\]
This leaves the stray field energy to be estimated.

Note that
\[
\left|\tilde{m}_1 - m_1\right| \le 1 - m_1 \le (\phi - \ell_\pm)^2 \le \omega^2
\]
in $\R \setminus [-R, R]$, whereas
\[
\begin{split}
\left|\tilde{m}_1' - m_1'\right| & = |\big(\eta' (\phi - \ell_\pm) + \eta \phi'\big) \sin \tilde{\phi} - \phi' \sin \phi| \\
& \le |\phi'| |\sin \phi - \eta \sin \tilde{\phi}| + |\eta'| |\phi - \ell_\pm| |\sin \tilde{\phi}| \\
& \le |\phi'| |\phi - \ell_\pm| + |\eta'| (\phi - \ell_\pm)^2 \\
& \le \omega \sigma + \frac{2\omega^2}{R}.
\end{split}
\]
Here we have used the inequality $|\sin \tilde \phi|=|\sin\big(\eta(\phi-\ell_\pm)\big)|\leq |\eta| |\phi-\ell_\pm|$ together with $|\sin t - \eta \sin (\eta t)|\leq |\sin  t|\leq |t|$ for all $0\leq \eta\leq 1$ and $t\in [-\pi/2, \pi/2]$ (applied to $t:=\phi-\ell_\pm$).
Thus by interpolation between $L^2(\R)$ and $\dot{H}^1(\R)$, we obtain
\[
\left\|\tilde{m}_1 - m_1\right\|_{\dot{H}^{1/2}(\R)}^2 \le C_1 \left(\int_R^\infty \omega^4 \, dx_1\right)^{1/2} \left(\int_R^\infty \left(\frac{\omega^4}{R^2} + \omega^2 \sigma^2\right) \, dx_1\right)^{1/2}
\]
for some universal constant $C_1$.

Now recall that $v \in \dot{H}^1(\R_+^2)$ is the harmonic extension of $m_1 - 1$ to $\R_+^2$ and
let $\tilde{v} \in \dot{H}^1(\R_+^2)$ be the harmonic extension of $\tilde{m}_1 - 1$.
Then
\[
\int_{\R_+^2} |\nabla \tilde{v}|^2 \, dx = \int_{\R_+^2} |\nabla v|^2 \, dx + \int_{\R_+^2} |\nabla v - \nabla \tilde{v}|^2 \, dx - 2 \int_{\R_+^2} \nabla v \cdot (\nabla v - \nabla \tilde{v}) \, dx.
\]
We know that
\[
\begin{split}
\int_{\R_+^2} |\nabla v - \nabla \tilde{v}|^2 \, dx & = \left\|m_1 - \tilde{m}_1\right\|_{\dot{H}^{1/2}(\R)}^2 \\
& \le C_1 \left(\int_R^\infty \omega^4 \, dx_1\right)^{1/2} \left(\int_R^\infty \left(\frac{\omega^4}{R^2} + \omega^2 \sigma^2\right) \, dx_1\right)^{1/2}.
\end{split}
\]
Moreover, an integration by parts gives
\[
- 2 \int_{\R_+^2} \nabla v \cdot (\nabla v - \nabla \tilde{v}) \, dx = 2 \int_{-\infty}^\infty (m_1 - \tilde{m}_1) \dd{v}{x_2} \, dx_1 \le 4\int_R^\infty \omega^2 \tau \, dx_1.
\]
Hence
\begin{multline*}
\frac{1}{2} \int_{\R_+^2} |\nabla \tilde{v}|^2 \, dx \le\frac{1}{2} \int_{\R_+^2} |\nabla v|^2 \, dx \\
+ \frac{C_1}{2} \left(\int_R^\infty \omega^4 \, dx_1\right)^{1/2} \left(\int_R^\infty \left(\frac{\omega^4}{R^2} + \omega^2 \sigma^2\right) \, dx_1\right)^{1/2} + 2\int_R^\infty \omega^2 \tau \, dx_1.
\end{multline*}
The desired inequality then follows.
\end{proof}

When we apply Lemma \ref{lem:localisation}, we will consider a fixed $d \in \Z \pm \{0, \alpha/\pi\}$ and
a profile $m$ that minimises $E_h$ in $\A_h(d)$. Then by Lemma \ref{lem:no_unaccounted_walls},
there exist $a_1, \dotsc, a_N \in \R$ with $a_1 < \dotsb < a_N$ and $N \le 2|d| + 1$
such that $m_2(a_n) = 0$ for $n = 1, \dotsc, N$ and $m_2 \not= 0$ elsewhere. Thus
Proposition \ref{prop:improved_decay}
applies between any pair of points $(a_n, a_{n + 1})$ and also in $(-\infty, a_1)$ and
in $(a_N, \infty)$. If $\phi$ is a lifting of $m$, choosing
\begin{align}
\nonumber
\omega(x_1) & = \max\{|\phi(x_1) - \ell_+|, |\phi(-x_1) - \ell_-|\} \\
\label{star2}
\sigma(x_1) & = \max\{|\phi'(x_1)|, |\phi'(-x_1)|\}, \\
\nonumber
\tau(x_1) & = \max\left\{\left|\dd{v}{x_2}(x_1, 0)\right|, \left|\dd{v}{x_2}(-x_1, 0)\right|\right\},
\end{align}
we then find, by Lemma \ref{lem:preliminary_decay_estimate}, that $\omega(x_1) \le \frac{\pi}{2}$
for $|x_1|$ large enough. Moreover, we then obtain $(\phi(x_1) - \ell_\pm)^2 \le \pi(1 - \cos \phi(x_1)) = \pi(1 - m_1(x_1))$
for $|x_1|$ large enough.\footnote{Note that $t^2/\pi\leq 1-\cos t\leq t^2/2$ for every $t\in [0, \pi/2]$.}
We write $I = (-2R, -R) \cup (R, 2R)$. Then Proposition \ref{prop:improved_decay} implies
\begin{equation} \label{eqn:decay1}
\begin{split}
\int_R^{2R} \omega^2 \, dx_1 & \le \pi \int_I (1 - m_1) \, dx_1 \\
& \le \pi \left(2R \int_I (1 - m_1)^2 \, dx_1\right)^{1/2} \le C_1 \sqrt{\frac{\log R}{R}}
\end{split}
\end{equation}
and
\begin{equation}
\begin{split}
\int_R^{2R} \omega \sigma \, dx_1 & \le \sqrt{\pi} \int_I \sqrt{1 - m_1} |m'| \, dx_1 \\
& \le \sqrt{\pi} \left(2R \int_I (1 - m_1)^2 \, dx_1\right)^{1/4} \left(\int_I |m'|^2 \, dx_1\right)^{1/2} \\
& \le \frac{C_1}{R^{5/4}} (\log R)^{3/4}
\end{split}
\end{equation}
for a constant $C_1 = C_1(d)$, provided that $R$ is sufficiently large.
Moreover,
\begin{align}
\int_R^\infty \omega^4 \, dx_1 & \le \pi^2 \int_{\R \setminus (-R, R)} (m_1 - 1)^2 \, dx_1 \le \frac{C_1}{R^2} \log R, \\
\int_R^\infty \omega^2 \sigma^2 \, dx_1 & \le \pi \int_{\R \setminus (-R, R)} \left(\frac{(m_1 - 1)^2}{R} + R|m'|^4\right) \, dx_1 \le \frac{C_1}{R^3} \log R, \\
\label{eqn:decay5}
\int_R^\infty \omega^2 \tau \, dx_1 & \le \pi \int_{\R \setminus (-R, R)} \left(R^{-1/2} (m_1 - 1)^2 + R^{1/2}\left(\dd{v}{x_2}\right)^2\right) \, dx_1 \le \frac{C_1}{R^{5/2}} \log R.
\end{align}

In the case $h < 1$, we use \cite[Proposition 2.1]{Ignat-Moser:17} instead of Lemma \ref{lem:localisation},
which gives rise to inequality \eqref{eqn:old_localisation}.
In this case, we know that $|\phi - \ell_\pm| \le c |m_1 - h|$ and $|\phi'| \le c |m_1'|$,
with a constant $c$ depending on $h$, when $|x_1|$ is sufficiently large. This gives rise to
\begin{align}
\label{eqn:decay6}
\int_R^\infty \omega^2 \, dx_1 & \le c^2 \int_{\R \setminus (-R, R)} (m_1 - h)^2 \, dx_1 \le \frac{C_2}{R^2} \log R, \\
\int_R^\infty \sigma^2 \, dx_1 & \le c^2 \int_{\R \setminus (-R, R)} |m_1'|^2 \, dx_1 \le \frac{C_2}{R^4} \log R, \\
\label{eqn:decay8}
\int_R^\infty \omega \tau \, dx_1 & \le c \int_{\R \setminus (-R, R)} \left(R^{-1/2} (m_1 - h)^2 + R^{1/2}(u')^2\right) \, dx_1 \le \frac{C_2}{R^{5/2}} \log R
\end{align}
for a constant $C_2 = C_2(d, h)$, provided that $R$ is sufficiently large.
These estimates are not quite as good as the ones derived in \cite{Ignat-Moser:17},
but they are sufficient for our purpose and they avoid a significant part of the
previous analysis. This conclusion can be summarized as follows.

\begin{corollary}
\label{cor:local}
Let $h\in [0,1]$ and $d\in \Z\pm \{0, \alpha/\pi\}$. Suppose that $m$ is a minimizer of $\E_h(d)$ in $\A_h(d)$.
Then there exist 
$R_0\geq 1$ and $\beta>0$ such that for every $R\geq R_0$, there exists $\tilde m \in \A_h(d)$
that is locally constant in $\R\setminus [-R, R]$ and satisfies
\[
E_h(\tilde m)\leq \E_h(d)+\frac1{R^{2+\beta}}
\]
and $|\tilde m_1-h|\leq |m_1-h|$ everywhere in $\R$.
\end{corollary}

\section{Proofs of the main results}

\subsection{Existence and non-existence of minimisers} \label{sect:existence/nonexistence}

For the proof Theorem \ref{thm:existence},
we can now largely use the arguments from our previous paper \cite{Ignat-Moser:17},
but we replace the previous decay estimates by Proposition \ref{prop:improved_decay}
and the previous $L^1$-estimates by Proposition \ref{prop:L^1-estimate}.
The main task is to estimate the stray field energy for potential minimisers of $E_h$.
To this end, we divide $m_1 - h$ into several pieces, each of which is
either nonpositive or nonnegative in $\R$.
As we may express the stray field energy in terms of the $\dot{H}^{1/2}$-inner product,
the following estimate, given as Lemma 4.3 in \cite{Ignat-Moser:17}, is particularly
useful here.

\begin{lemma} \label{lem:H^{1/2}}
Let $f, g \in \dot{H}^{1/2}(\R)$ be nonnegative functions and suppose that there exists
$R > 0$ with $\supp f \subseteq [- 2R, - R]$ and $\supp g \subseteq [R, 2R]$.
Then
\[
-\frac{1}{4\pi R^2} \|f\|_{L^1(\R)} \|g\|_{L^1(\R)} \le \scp{f}{g}_{\dot{H}^{1/2}(\R)} \le -\frac{1}{16\pi R^2} \|f\|_{L^1(\R)} \|g\|_{L^1(\R)}.
\]
\end{lemma}

Using this and the above estimates, we can now examine what happens if a magnetisation profile
$m$ is split into two or more parts or if several profiles are combined. For this purpose,
we use the following notion.

\begin{definition}
Fix $h \in (0, 1]$ and let $D_h = (\N \pm \{0, \alpha/\pi\}) \cup \{\alpha/\pi\}$.
Suppose that $d \in D_h$.
For $J \in \N$ and $d_1, \dotsc, d_J \in D_h$, we say that $(d_1, \dotsc, d_J)$ is a
\emph{partition} of $d$ if
\begin{itemize}
\item $d = d_1 + \dotsb + d_J$ and
\item for all $j, k \in \{1, \dotsc, J\}$ with $j < k$, if $d_j \not\in \N$ and $d_k \not\in \N$ but
$d_{j + 1}, \dotsc, d_{k - 1} \in \N$, then $d_j + d_k \in \N$.
\end{itemize}
We say that the partition is \emph{trivial} if $J = 1$.
\end{definition}

The conditions in the definition guarantee that profiles $m^{(j)} \in \A_h(d_j)$
can be combined to form a profile in $\A_h(d)$, although for $d_j \in \N$, it may be necessary
to reverse the orientation of $\R$ as well as $\Ss^1$ and consider $m_1^{(j)}(-x_1)$ instead of $m_1^{(j)}$
and $-m_2^{(j)}(-x_1)$ instead of $m_2^{(j)}$.

One of the key ingredients for the proof of Theorem \ref{thm:existence} is a concentration-compactness
principle that allows to prove existence of minimisers in $\A_h(d)$ under the assumption
that any nontrivial partition of $d$ will give rise to a larger total energy. A special case was
formulated in our previous paper \cite[Theorem 6.1]{Ignat-Moser:17}. The statement can
easily be extended as follows.

\begin{theorem} \label{thm:concentration-compactness}
Suppose that $d \in D_h$ is such that
\[
\E_h(d) < \sum_{k = 1}^J \E_h(d_k)
\]
for all nontrivial partitions $(d_1, \dotsc, d_J)$ of $d$. Then $E_h$ attains its infimum in $\A_h(d)$.
\end{theorem}

\begin{proof} 
For completeness, we adapt the arguments presented in the proof of \cite[Theorem~6.1]{Ignat-Moser:17}. Consider a minimising sequence $(m^j)_{j \in \N}$ of $E_h$
in $\A_h(d)$. Up to translation, we can assume that 
\be
\label{mzero}
m^j(0) \in \{ (\pm 1, 0)\} \text{ if $h<1$} \quad \text{and} \quad m^j(0)=(-1,0) \text{ if $h=1$}
\ee
for every $j \in \N$.\footnote{In the proof of \cite[Theorem~6.1]{Ignat-Moser:17}, we further used a symmetrisation argument when choosing the minimizing sequence $(m^j)_{j \in \N}$, but this is not necessary here.}
Furthermore, as we may choose a subsequence if necessary, we may assume that $m^j \rightharpoonup m$
weakly in $H^{1}_\loc(\R; \Ss^1)$ and locally uniformly in $\R$ for some $m \in H^{1}_\loc(\R; \Ss^1)$
and
\[
\int_{-2j}^{2j} |m^j - m|^2 \, dx_1 \le \frac{1}{j^5} \quad \text{for all } j\in \N.
\]
Then $m(0)$ satisfies \eqref{mzero}, and
the lower semicontinuity of the energy with
respect to such convergence implies
\[
E_h(m) \le \liminf_{j \to \infty} E_h(m^j) = \E_h(d).
\]
In particular, we have $\lim_{x_1 \to \pm \infty} m_1(x_1) = h$,
and the winding number $\tilde{d} = \deg(m)$ is well-defined and belongs to $\Z + \{0, \pm \alpha/\pi\}$; moreover, 
$$\lim_{j \to \infty} \|m_1^j - h\|_{L^\infty([-2j, -j] \cup [j, 2j])} = 0 \quad \textrm{and}\quad \lim_{j \to \infty} \int_{-2j}^{2j} (m^j)^\perp \cdot (m^j)' \, dx_1 = 2\pi \tilde{d}.$$
The aim is to show that $\tilde{d} = d$, which entails that $m$ is a minimiser of $E_h$ in $\A_h(d)$.

The idea is to split the map $m^j$ for each $j\in \N$ by cutting off a left part $\hat{m}^{-j}$, a middle part  $\tilde{m}^{j}$ and a right part $\hat{m}^{j}$ such that $\hat{m}^{-j}=m^j$ in $(-\infty, -2j)$ and constant in $(-7j/4, +\infty)$, $\tilde{m}^{j}=m^j$ in $(-j,j)$ and constant in $(-\infty, -5j/4)\cup (5j/4, +\infty)$, while 
$\hat{m}^{j}=m^j$ in $(2j, +\infty)$ and constant in $(-\infty, 7j/4)$. Then
$\deg(\tilde m^j)=\tilde d$, and if we denote $d^{\pm j}=\deg(\hat{m}^{\pm j})$, we have $\tilde d+d^{-j}+d^j=d$ for every sufficiently large $j\in \N$. Following the argument in the proof of 
\cite[Theorem~6.1]{Ignat-Moser:17}, we obtain
$$\limsup_{j \to \infty} \left(E_h(\tilde{m}^j) + E_h(\hat{m}^j) + E_h(\hat{m}^{-j})\right) \le \E_h(d).$$
In particular, we deduce that the two sequences $(\E_h(d^{\pm j}))_{j \in \N}$
are bounded; by Lemma~\ref{lem:no_unaccounted_walls}, it follows that the two sequences $(d^{\pm j})_{j \in \N}$ are bounded. Therefore, we can extract a subsequence, such that after relabelling, those sequences are constant, i.e., $d^{-j}=d^-$ and 
$d^j=d^+$ for two degrees $d^{\pm}\in \Z + \{0, \pm \alpha/\pi\}$. Note that
all the pairs $(d^-, \tilde{d})$, $(\tilde{d},d^+$), $(d^-, \tilde{d} + d^+)$, and $(d^+, \tilde{d} + d^-)$
comprise compatible neighbouring degrees (see Remark \ref{rem:degree}). Moreover, by Lemma \ref{lem:no_unaccounted_walls}, combined with \eqref{mzero}, we know that $\limsup_{j \to \infty} E_h(\tilde{m}^j)\geq \max\{C, \E_h(\tilde{d})\}$ for some $C>0$ depending only on $h$. Therefore, it follows that
\be
\label{id_degre}
\E_h(d^+)+\E_h(d^-)+\max\{C, \E_h(\tilde{d})\}\le \E_h(d) \quad \textrm{and} \quad \tilde d+d^{-}+d^+=d.
\ee
We claim that the above relation entails $d=\tilde d$. This follows in several steps.

\begin{severalsteps}

\step{we prove that $d^+, d^-\geq 0$} Assume by contradiction that $d^-<0$ (the other case  $d^+<0$ can be treated identically). In particular, by \eqref{id_degre}, $d^++\tilde d=d-d^->d$. As $\tilde d$ and $d^+$ are compatible neighbouring degrees, Remark \ref{rem:degree} implies that $\E_h(\tilde{d})+\E_h(d^+)\geq \E_h(\tilde{d}+d^+)$. We distinguish two cases.

\begin{itemize}
\item If $0<h<1$ and $(d^++\tilde d,d)=(\ell+1-\alpha/{\pi}, \ell+\alpha/\pi)$ for some $\ell\in \N\cup\{0\}$, then $d^-=2\alpha/\pi-1\in \Z + \{0, \pm \alpha/\pi\}$. Thus
$\alpha=\pi/3$ (as $\alpha\in (0, \pi/2)$), $d^-=-1/3$, and $\tilde d+d^+=\ell+2/3$, which contradicts the compatibility of the neighbouring transitions of degree $d^-$ and $\tilde d+d^+$. 

\item Otherwise, the monotonicity in Remark \ref{rem:degree} applies and we conclude that
\[
\E_h(\tilde{d})+\E_h(d^+)\geq \E_h(\tilde{d}+d^+)\geq \E_h(d)\stackrel{\eqref{id_degre}}{\geq}
\E_h(\tilde{d})+\E_h(d^+)+\E_h(d^-).
\]
In particular, $\E_h(d^-)=0$. By Lemma \ref{lem:no_unaccounted_walls}, this would imply $d^-=0$, which contradicts the assumption $d^-<0$.
\end{itemize}

\step{we prove that $\tilde d+d^+, \tilde d+d^-> 0$}Assume by contradiction that $d^-+\tilde d\leq 0$ (the other case  $d^++\tilde d\leq 0$ can be treated identically). By \eqref{id_degre}, $d^+=d- (d^-+\tilde d)\geq d$.
We distinguish two cases.

\begin{itemize}
\item If $0<h<1$ and $(d^+,d)=(\ell+1-\alpha/{\pi}, \ell+\alpha/\pi)$ for some $\ell\in \N\cup\{0\}$, then $d^-+\tilde d=2\alpha/\pi-1\in \Z + \{0, \pm \alpha/\pi\}$. Thus
$\alpha=\pi/3$, $d^+=\ell+2/3$, and $\tilde d+d^-=-1/3$, which contradicts the compatibility of the neighbouring transitions of degree $d^+$ and $\tilde d+d^-$. 

\item Otherwise, the monotonicity in Remark \ref{rem:degree} applies and gives
\[
\E_h(d^+)\geq \E_h(d)\stackrel{\eqref{id_degre}}{\geq}
\E_h(d^+)+C,
\]
which is impossible as $C>0$.
\end{itemize}

\step{we prove that $\tilde d> 0$} Assume by contradiction that  $\tilde d\leq 0$. As $\tilde d$ and 
$d^+$ are compatible neighbouring degrees, Remark \ref{rem:degree} implies that $\E_h(d^+)+\E_h(\tilde d)
\geq \E_h(d^++\tilde d)$. By \eqref{id_degre}, it follows that $\E_h(d)\geq \E_h(d^++\tilde d)+\E_h(d^-)$. 
As $d^-\geq 0$ (by Step 1) and $d^++\tilde d>0$ (by Step 2), the assumption of the theorem implies that $(d^-, d^++\tilde d)$ is \emph{not}
a partition of $d$, i.e., that $d^-=0$. This in turn implies that $d^-+\tilde d=\tilde d\leq 0$, which contradicts Step 2.

\step{we conclude that $d=\tilde d$} If this were not the case, then Steps 1 and 3, combined with the compatibility of neighbouring transitions of degree $d^-, \tilde d$ and $d^+$ (see Remark \ref{rem:degree}),
would imply that $(d^-, \tilde d, d^+)$ or $(d^-, \tilde d)$ or $(\tilde d, d^+)$ is a nontrivial partition
of $d$. By the hypothesis of the theorem, however, this would contradict \eqref{id_degre}.
Therefore, we see that $d=\tilde d$. 
\end{severalsteps}
\end{proof}

If we want to use Theorem \ref{thm:concentration-compactness}, we need to verify the inequality
in the hypothesis.
We first study what happens if the profiles
of several \emph{minimisers} of $E_h$ (for their respective winding numbers) are combined.
This is the counterpart of \cite[Theorem 7.2]{Ignat-Moser:17} for higher winding numbers.

\begin{proposition} \label{prop:inequality_composite_walls}
For any $\ell \in \N$ there exists $H \in (0, 1)$ such that the
following holds true for all $h \in (H, 1]$. Suppose that $d = \ell - \alpha/\pi$.
Let $(d_1, \dotsc, d_J)$
be a nontrivial partition of $d$ such that there exists a minimizer $m^{(j)} \in \A_h(d_j)$
with $E_h(m^{(j)}) = \E_h(d_j)$ for all $j = 1, \dotsc, J$. Then
\[
\E_h(d) < \sum_{j = 1}^J \E_h(d_j).
\]
\end{proposition}

\begin{proof}
We may assume that
\[
\lim_{x_1 \to + \infty} m^{(j)}(x_1) = \lim_{x_1 \to - \infty} m^{(j + 1)}(x_1), \quad j = 1, \dotsc, J - 1.
\]
(This is because in the case $d_j\in \N$, we may replace $m_1^{(j)}(x_1)$ by $m_1^{(j)}(-x_1)$ and $m_2^{(j)}(x_1)$ by $-m_2^{(j)}(-x_1)$ for some values of $j$.) It is easy to see that there exists a constant $C_1 = C_1(\ell)$ such that
\[
\sum_{j = 1}^J E_h(m^{(j)}) \le C_1.
\]
By Corollary \ref{cor:local}, we may modify $m^{(j)}$
such that the support of $m_1^{(j)} - h$ becomes
compact, while changing the energy only by a small amount. More specifically, we find two numbers 
$R_0\geq 1$ and $\beta > 0$ such that for all $R \ge R_0$,
there exist $\tilde{m}^{(j)} \in \A_h(d_j)$ with
$\tilde{m}_1^{(j)} = h$ outside of $[-R, R]$,
while at the same time,
\[
E_h(\tilde{m}^{(j)}) \le E_h(m^{(j)}) + \frac{1}{R^{2 + \beta}}
\]
and $|\tilde{m}_1^{(j)} - h| \le |m_1^{(j)} - h|$ everywhere for $j = 1, \dotsc, J$.

Because $d \in \N- \alpha/\pi$, it follows that $d_1, d_J \in \N - \{0, \alpha/\pi\}$.
Since $\tilde{m}^{(j)} \in \A_h(d_j)$, this means that it contains a full transition on the half circle $\{z\in \Ss^1\, :\, z_1\leq 0\}$.
Furthermore, as $E_h(\tilde{m}^{(j)}) \le C_1+1$,
it is easy to
see\footnote{\label{foot}If $f=-(\tilde{m}_1^{(j)})_-$, then we may assume that $f(0)=1$. Since $\|f'\|_{L^2(\R)}\leq \sqrt{2C_1+1}$, then $f(x_1)\geq f(0)- \|f'\|_{L^2(\R)} |x_1|^{1/2} \geq 1-\sqrt{2C_1+1} |x_1|^{1/2}$, which proves \eqref{eqn:negative_part1}.}
that there exists a constant $c_0 = c_0(\ell) > 0$ with
\begin{equation} \label{eqn:negative_part1}
\bigl\|\bigl(\tilde{m}_1^{(j)} - h\bigr)_-\bigr\|_{L^1(\R)}\geq \bigl\|(\tilde{m}_1^{(j)})_-\bigr\|_{L^1(\R)} \ge c_0
\end{equation}
for $j = 1, J$. (This estimate is essential and the arguments work only for the degree $d\in \N-\alpha/\pi$ due to the ``sandwich" configuration created by the outermost transitions corresponding to $j=1$ and $j=J$.)

On the other hand, Proposition \ref{prop:L^1-estimate} (for $p = \frac{5}{3}$) implies that
there exists $C_2 = C_2(\ell)$ satisfying
\begin{equation} \label{eqn:positive_part}
\bigl\|\bigl(\tilde{m}_1^{(j)} - h\bigr)_+\bigr\|_{L^1(\R)} \le C_2 \alpha^{2/11}
\end{equation}
and
\begin{equation} \label{eqn:negative_part2}
\bigl\|\bigl(\tilde{m}_1^{(j)} - h\bigr)_-\bigr\|_{L^1(\R)} \le C_2
\end{equation}
for all $j = 1, \dotsc, J$.

Now we may define $m \colon \R \to \Ss^1$ by $m(x_1) = \tilde{m}^{(j)}(x_1 - 6jR)$
for $6jR - R \le x_1 \le 6jR + R$, where $j = 1, \dotsc, J$, and $m(x_1) = (\cos \alpha, \pm \sin \alpha)$
elsewhere (so that $m$ is continuous everywhere).
Then $m \in \A_h(d)$. It is clear that
\[
\int_{-\infty}^\infty \left(|m'|^2 + (m_1 - h)^2\right) \, dx_1 = \sum_{j = 1}^J \int_{-\infty}^\infty \left(|(\tilde{m}^{(j)})'|^2 + (\tilde{m}_1^{(j)} - h)^2\right) \, dx_1.
\]
Next we note that
\[
m_1(x_1) - h = \sum_{j = 1}^J \bigl(\tilde{m}_1^{(j)}(x_1 - 6jR) - h\bigr)_+ + \sum_{j = 1}^J \bigl(\tilde{m}^{(j)}_1(x_1 - 6jR) - h\bigr)_-.
\]
If $j \neq j'$, then the supports of the functions $(\tilde{m}^{(j)} - h)_\pm$ and $(\tilde{m}^{(j')} - h)_\pm$ are contained in intervals of length $2R$ each and are separated by at least $4R$. Therefore, we may apply Lemma \ref{lem:H^{1/2}} to estimate
\[
\scp{(\tilde{m}^{(j)} - h)_\pm}{(\tilde{m}^{(j')} - h)_\pm}_{\dot{H}^{1/2}(\R)}
\]
for any
such pair. Because of inequalities \eqref{eqn:negative_part1},
\eqref{eqn:positive_part}, and \eqref{eqn:negative_part2}, we obtain a constant $C_3 = C_3(\ell) > 0$
such that
\[
\|m_1 - h\|_{\dot{H}^{1/2}(\R)}^2 \le \sum_{j = 1}^J \bigl\|\tilde{m}_1^{(j)} - h\bigr\|_{\dot{H}^{1/2}(\R)}^2 + \frac{C_3 \alpha^{2/11}}{R^2} - \frac{1}{C_3 R^2}.
\]
Therefore,
\[
E_h(m) \le \sum_{j = 1}^J E_h(m^{(j)}) + \frac{J}{R^{2 + \beta}} + \frac{C_3 \alpha^{2/11}}{2R^2} - \frac{1}{2C_3 R^2}.
\]
If $\alpha$ is chosen sufficiently small and $R$ sufficiently large, then the above error becomes negative, which leads to the desired inequality.
\end{proof}

Because the preceding result only applies to winding numbers where minimisers
exist, we also need some information for the other cases.

\begin{proposition} \label{prop:inequality_repulsion}
Suppose that $h \in (0, 1]$ and $d \in D_h$. Then there exists
a partition $(d_1, \dotsc, d_J)$ of $d$ such that
\[
\E_h(d) \ge \sum_{j = 1}^J \E_h(d_j)
\]
and $\E_h(d_j)$ is attained for all $j = 1, \dotsc, J$.
\end{proposition}

\begin{proof}
The set $D_h$ allows a proof by induction. The statement is true for $d = \alpha/\pi$ and
$d = 1 - \alpha/\pi$, because $\E_h(\alpha/\pi)$ and $\E_h(1 - \alpha/\pi)$ are attained \cite{Me1, Chermisi-Muratov:13, Ignat-Moser:17} and the trivial partition has the desired property.

Now fix $d \in D_h$ and assume that the statement is proved for all numbers $d' \in D_h$
with $d' < d$. If $\E_h(d)$ is attained, then we use the trivial partition again. Otherwise,
Theorem \ref{thm:concentration-compactness} implies that
there exists a nontrivial partition $(d_1, \dotsc, d_J)$ of $d$ such that
\[
\E_h(d) \ge \sum_{j = 1}^J \E_h(d_j).
\]
Then $d_j < d$ for all $j = 1, \dotsc, J$, and therefore, the induction assumption applies.
Thus for any $j \in \{1, \dotsc, J\}$, there exists a partition $(d_{j1}, \dotsc, d_{jK_j})$ of $d_j$
such that
\[
\E_h(d_j) \ge \sum_{k = 1}^{K_j} \E_h(d_{jk})
\]
and every $\E_h(d_{jk})$ is attained. Combining all the resulting partitions, we obtain a partition of $d$ with
the desired properties. (If $d_j \in \N$ for some $j = 1, \dotsc, J$, we may need to reorder
$d_{j1}, \dotsc, d_{jK_j}$ in order to achieve another partition of $d$, but that does not invalidate the argument.)
\end{proof}

\begin{proof}[Proof of Theorem \ref{thm:existence}]
Let $d \in \N - \alpha/\pi$.
According to Theorem \ref{thm:concentration-compactness}, it suffices to show
that
\[
\E_h(d) < \sum_{j = 1}^J \E_h(d_j)
\]
for any nontrivial partition $(d_1, \dotsc, d_J)$ of $d$.

Assuming that $\E_h(d_j)$ is attained for all $j = 1, \dotsc, J$, the inequality follows from
Proposition \ref{prop:inequality_composite_walls}, provided that $\alpha$ is sufficiently small.
If there is $j \in \{1, \dotsc, J\}$ such
that $\E_h(d_j)$ is \emph{not} attained, then we replace all such $d_j$ by a partition of $d_j$
with the properties of Proposition \ref{prop:inequality_repulsion}. Eventually, we are
in a situation where we can use Proposition \ref{prop:inequality_composite_walls}, and then the claim follows.
\end{proof}

\begin{proof}[Proof of Theorem \ref{thm:non-existence}]
We argue by contradiction here. Let $h \in (0, 1)$, to be chosen sufficiently close to $1$ eventually.
Suppose that $\ell \in \N$ and $d = \ell$ or $d = \ell + \alpha/\pi$.
Suppose that $E_h$ had a minimiser $m$ in $\A_h(d)$. We may assume without loss of generality that
\[
\lim_{x \to -\infty} m(x_1) = (\cos \alpha, -\sin \alpha).
\]
(This is automatic if $d = \ell + \alpha/\pi$. If $d = \ell$, we may have to replace $m_1(x_1)$
by $m_1(-x_1)$ and $m_2(x_1)$ by $-m_2(-x_1)$.) Choose a continuous lifting $\phi \colon \R \to \R$ such that
$m = (\cos \phi, \sin \phi)$. Then there exists $b \in \R$ such that
\[
\int_{-\infty}^b \phi'(x_1) \, dx_1 = 2\alpha.
\]
Indeed, there exists a largest number $b$ with this property, and if we choose
this one, then there also exist $a, c \in \R$ such that $b < a < c$ and
such that $m_1(a) = -1$, $m_1(c) = h$, and $m_1 < h$ in $(b, c)$.

An easy construction shows that there exists a number $C_1 = C_1(\ell)$ satisfying
$\E_h(d) \le C_1$. Since $m$ is assumed to be an energy minimiser in $\A_h(d)$, this means
that $E_h(m) \le C_1$. This implies a bound for $m_1$ in $C^{0, 1/2}([b, c])$, and it follows as in the footnote \ref{foot}
that
\be
\label{staru}
\int_b^c (h - m_1) \, dx_1 \ge C_2
\ee
for some constant $C_2 = C_2(\ell) > 0$.

Define the function
\[
m_1^+(x_1) = \begin{cases}
m_1(x_1) & \text{if $x_1 < b$ and $m_1(x_1) > h$}, \\
h & \text{otherwise}.
\end{cases}
\]
Then $m_1^+-h\geq 0$ everywhere. Further define
\[
m_1^-(x_1) = \begin{cases}
m_1(x_1) & \text{if $x_1 \ge b$ or $m_1(x_1) \le h$}, \\
h & \text{otherwise}.
\end{cases}
\]
Then $m_1 = m_1^+ + m_1^- - h$ and $(m_1^+ - h)(m_1^- - h) = 0$ everywhere. There exist $m_2^+, m_2^- \colon \R \to [-1, 1]$
such that
\[
m^+ = (m_1^+, m_2^+) \in \A_h(\alpha/\pi) \quad \text{and} \quad m^- = (m_1^-, m_2^-) \in \A_h(d - \alpha/\pi)
\]
by the choice of $b$ and the above definitions.

Now we compute
\[
E_h(m) = E_h(m^+) + E_h(m^-) + \scp{m_1^+}{m_1^-}_{\dot{H}^{1/2}(\R)}.
\]
Hence
\[
\E_h(d) \ge \E_h(\alpha/\pi) + \E_h(d - \alpha/\pi) + \scp{m_1^+}{m_1^-}_{\dot{H}^{1/2}(\R)}.
\]

Next we wish to estimate the quantity
\[
\scp{m_1^+}{m_1^-}_{\dot{H}^{1/2}(\R)} = \frac{1}{2\pi} \int_{-\infty}^\infty \int_{-\infty}^\infty \frac{(m_1^+(s) - m_1^+(t))(m_1^-(s) - m_1^-(t))}{(s - t)^2} \, ds \, dt.
\]
To this end, we first observe that $m_1^\pm$ may be replaced by $m_1^\pm - h$ without
changing the double integral. Furthermore, the construction guarantees that $m_1^+ - h = 0$ in $(b, \infty)$ and,
as mentioned before, that
$(m_1^+ - h)(m_1^- - h) = 0$ everywhere. Therefore,
\[
\begin{split}
\scp{m_1^+}{m_1^-}_{\dot{H}^{1/2}(\R)} & = -\frac{1}{\pi} \int_{-\infty}^\infty \int_{-\infty}^\infty \frac{(m_1^+(s) - h)(m_1^-(t) - h)}{(s - t)^2} \, ds \, dt \\
& = -\frac{1}{\pi} \int_{-\infty}^b \int_{-\infty}^b \frac{(m_1^+(s) - h)(m_1^-(t) - h)}{(s - t)^2} \, ds \, dt \\
& \quad -\frac{1}{\pi} \int_b^c \int_{-\infty}^b \frac{(m_1^+(s) - h)(m_1^-(t) - h)}{(s - t)^2} \, ds \, dt \\
& \quad -\frac{1}{\pi} \int_c^\infty \int_{-\infty}^b \frac{(m_1^+(s) - h)(m_1^-(t) - h)}{(s - t)^2} \, ds \, dt.
\end{split}
\]
If $t \le b$, then $m_1^-(t) - h \le 0$, whereas $m_1^+(s) - h \ge 0$ everywhere. Hence
\[
-\frac{1}{\pi} \int_{-\infty}^b \int_{-\infty}^b \frac{(m_1^+(s) - h)(m_1^-(t) - h)}{(s - t)^2} \, ds \, dt \ge 0.
\]
If $t \in (b, c)$ and $s \le b$, then $(s - t)^2 \le (s - c)^2$. Furthermore, in this case,
we still have the inequalities $m_1^+(s) - h \ge 0$ and $m_1^-(t) - h \le 0$. Hence
\[
\begin{split}
-\frac{1}{\pi} \int_b^c \int_{-\infty}^b \frac{(m_1^+(s) - h)(m_1^-(t) - h)}{(s - t)^2} \, ds \, dt & \ge \frac{1}{\pi} \int_{-\infty}^b \frac{m_1^+(s) - h}{(s - c)^2} \, ds \int_b^c (h - m_1^-(t)) \, dt \\
& \stackrel{\eqref{staru}}{\ge} \frac{C_2}{\pi} \int_{-\infty}^b \frac{m_1^+(s) - h}{(s - c)^2} \, ds.
\end{split}
\]
Finally, if $t \ge c$ and $s \le b$, then $(s - t)^2 \ge (s - c)^2$. Splitting $m_1^- - h$
into its positive part $(m_1^- - h)_+$ and its negative part $(m_1^- - h)_-$, we find that
\begin{multline*}
-\frac{1}{\pi} \int_c^\infty \int_{-\infty}^b \frac{(m_1^+(s) - h)(m_1^-(t) - h)}{(s - t)^2} \, ds \, dt \\
\ge -\frac{1}{\pi} \int_{-\infty}^b \frac{m_1^+(s) - h}{(s - c)^2} \, ds \int_c^\infty (m_1^-(t) - h)_+ \, dt.
\end{multline*}

Proposition \ref{prop:L^1-estimate} applies to $m$, because it is a minimiser in $\A_h(d)$.
This has consequences for $m_1^-$ as well; namely, there exists a number $C_3 = C_3(\ell)$
such that
\[
\int_c^\infty (m_1^- - h)_+ \, dt \le \int_{-\infty}^{+\infty} (m_1-h)_+\, dt\leq C_3 \alpha^{2/11}.
\]
Therefore, we obtain the inequality
\[
\E_h(d) \ge \E_h(\alpha/\pi) + \E_h(d - \alpha/\pi) + \frac{C_2 - C_3 \alpha^{2/11}}{\pi} \int_{-\infty}^b \frac{m_1^+(s) - h}{(s - c)^2} \, ds.
\]
If $1 - h$ is so small that $C_2 - C_3 \alpha^{2/11} > 0$, then the above error term is positive 
(as $m_1^+-h$ is nonnegative and not identically zero in $(-\infty, b)$). This
gives a direct contradiction to the subadditivity property in Remark \ref{rem:degree},
which asserts that $\E_h(d) \le \E_h(\alpha/\pi) + \E_h(d - \alpha/\pi)$ (as $\alpha/\pi$ and $d-\alpha/\pi$ are compatible neighbouring degrees).
\end{proof}

\subsection{The structure of minimisers}

For the proof of Theorem \ref{thm:structure}, we first need a bound for the
width of the profile of a minimiser $m$ of $E_h$ subject to a prescribed winding number.
We control this in terms of the distance between
any two points $a_1, a_2 \in \R$ with $m_1(a_1) = m_1(a_2) = -1$.

\begin{proposition} \label{prop:profile_width}
For every $\ell \in \N$ there exist $H \in (0, 1)$ and $\Lambda > 0$ such that for any
$h \in (H, 1)$ and any minimiser $m \in \A_h(\ell - \alpha/\pi)$ of $E_h$ in $\A_h(\ell - \alpha/\pi)$,
the inequality
\[
\diam \set{x_1 \in \R}{m(x_1) = (\pm 1, 0)} \le \Lambda
\]
holds true.
\end{proposition}

Before we prove this statement, however, we establish the following auxiliary result.

\begin{lemma} \label{lem:upper_semicontinuity}
The function $h \mapsto \E_h(\ell - \alpha/\pi)$ is upper semicontinuous
in $[0, 1]$.
\end{lemma}

It is not too difficult to show,  with arguments similar
to a previous paper \cite[Proposition 18]{Ignat-Moser:16}, that the function is actually continuous.
The above weaker statement, however, is sufficient for our purpose here.

\begin{proof}[Proof of Lemma \ref{lem:upper_semicontinuity}]
Let $\Phi_0$ be the set of all $\varphi \in H^1(\R)$ such that $\varphi \equiv 0$ in $(-\infty, -\sigma)$
and $\varphi \equiv 2\pi \ell$ in $[\sigma, \infty)$ for some $\sigma > 0$. Given $\varphi \in \Phi_0$, define
\[
\varphi_\alpha = \left(1 - \frac{\alpha}{\pi \ell}\right) \varphi + \alpha
\]
and
\[
f_\varphi(h) = E_h(\cos \varphi_\alpha, \sin \varphi_\alpha).
\]
As $\varphi$ has compact support, we deduce that $f_\varphi$ is a continuous function in $[0, 1]$ for every $\varphi \in \Phi_0$.
We claim that
\[
\E_h(\ell - \alpha/\pi) = \inf_{\varphi \in \Phi_0} f_\varphi(h).
\]
As the inequality $\E_h(\ell - \alpha/\pi) \le \inf_{\varphi \in \Phi_0} f_\varphi(h)$ is obvious, it suffices
 to prove the opposite inequality. To this end, if $m\in \A_h(\ell - \alpha/\pi)$, then we use the localisation in Lemma \ref{lem:localisation} for $h=1$ and \eqref{eqn:old_localisation} for $h<1$. More precisely, we consider the functions $\omega, \sigma$ and $\tau$ given in \eqref{star2} in terms of the lifting $\phi$ of $m$.
 
In the case $h<1$, we see that 
$\omega\in L^2(\R)$ (as $m_1-h\in L^2(\R)$),  $\sigma\in L^2(\R)$ (as $\phi'\in L^2(\R)$), and $\tau\in L^2(\R)$ (as $\frac{\partial v}{\partial x_2}(\blank, 0)$ is the Dirichlet-to-Neumann operator associated to $m_1-h$, see \cite[Section 1.6]{Ignat-Moser:17}, so that $\| \frac{\partial v}{\partial x_2}(\blank, 0)\|_{L^2(\R)}=\|m'_1\|_{L^2(\R)}$). 
Therefore, by \cite[Proposition 2.1]{Ignat-Moser:17}, for large $R>0$ we can find  a map $\tilde m_R\in \A_h(\ell - \alpha/\pi)$ that is constant outside $[-2R,2R]$ and $E_h(\tilde m_R)-E_h(m)\leq o(1)$ as $R\to \infty$, so the lifting of $\tilde m_R$ can be written as $\varphi_\alpha$ above.

In the case $h=1$, we apply the localisation in Lemma \ref{lem:localisation}, using an $L^4$-estimate for $\omega$ (corresponding to $m_1-1\in L^2(\R)$), as well as an $L^\infty$-estimate  for $\omega$ when estimating $\omega^2$ and $\omega \sigma$ in $(R, 2R)$ and $w^2 \sigma^2$ in $(R, \infty)$.

Thus $h \mapsto \E_h(\ell - \alpha/\pi)$ is an infimum of continuous functions and therefore upper semicontinuous.
\end{proof}

\begin{proof}[Proof of Proposition \ref{prop:profile_width}]
We argue by contradiction.
Suppose that there exists a sequence $h_i \nearrow 1$ such that we have a minimiser
$m^{(i)}$ of $E_{h_i}$ in $\A_{h_i}(\ell - \alpha_i/\pi)$ for every $i \in \N$,
where $\alpha_i = \arccos h_i\in (0,\pi/2)$, satisfying $\alpha_i\to 0$ and
\[
\diam \set{x_1 \in \R}{m^{(i)}(x_1) = ( \pm 1, 0)} \to \infty
\]
as $i \to \infty$. Recall that $E_{h_i}(m^{(i)})=\E_{h_i}(\ell-\alpha_i/\pi)\leq C(\ell)$ as $i \to \infty$.

By Lemma \ref{lem:no_unaccounted_walls}, there exist exactly $2\ell - 1$ values $x_1 \in \R$
with $m^{(i)}(x_1) = ( \pm 1, 0)$ for every $i \in \N$. We now want to arrange these
points into several groups such that the diameter of each group remains bounded,
but the distance between any two groups tends to infinity. Indeed,
after passing to a subsequence if necessary, we may find a number $N \ge 2$ such
that for every $i \in \N$, there exist $a_1^{(i)}, \dotsc, a_N^{(i)} \in \R$ with
$a_1^{(i)} < \dotsb < a_N^{(i)}$ and $m^{(i)}(a_n^{(i)}) = (\pm 1, 0)$ for $i \in \N$,
and such that furthermore,
\[
\lim_{i \to \infty} (a_{n + 1}^{(i)} - a_n^{(i)}) = \infty, \quad n = 1, \dotsc, N - 1,
\]
while
\[
\limsup_{i \to \infty} \, \sup \set{\min_{n = 1, \dotsc, N} |x_1 - a_n^{(i)}|}{x_1 \in \R \text{ with } m^{(i)}(x_1) = (\pm 1, 0)} < \infty.
\]

Define
\[
m_n^{(i)}(x_1) = m^{(i)}(x_1 - a_n^{(i)}), \quad i \in \N, \ n = 1, \dotsc, N.
\]
As $m^{(i)}$ has uniformly bounded energy as $i\to \infty$, each of the sequences $(m_n^{(i)})_{i \in \N}$ is bounded in $H^1((-R, R); \Ss^1)$
for every $R > 0$. Therefore, we may assume that $m_n^{(i)} \rightharpoonup \tilde{m}_n$
weakly in $H_\loc^1(\R; \Ss^1)$ and locally uniformly in $\R$ as $i \to \infty$. Then
\[
\begin{split}
E_1(\tilde{m}_n) & = \frac{1}{2} \lim_{R \to \infty} \Biggl(\int_{-R}^R \left(|\tilde{m}_n'|^2 + (\tilde{m}_{n1} - 1)^2\right) \, dx_1 \\
& \hspace{22mm} + \frac{1}{2\pi} \int_{-R}^R \int_{-R}^R \frac{(\tilde{m}_{n1}(s) - \tilde{m}_{n1}(t))^2}{(s - t)^2} \, ds \, dt\Biggr) \\
& \le \frac{1}{2} \lim_{R \to \infty} \liminf_{i \to \infty} \Biggl(\int_{-R}^R \left(|(m_n^{(i)})'|^2 + (m_{n1}^{(i)} - h_i)^2\right) \, dx_1 \\
& \hspace{33mm} + \frac{1}{2\pi} \int_{-R}^R \int_{-R}^R \frac{(m_{n1}^{(i)}(s) - m_{n1}^{(i)}(t))^2}{(s - t)^2} \, ds \, dt\Biggr).
\end{split}
\]
Hence
\[
\sum_{n = 1}^N E_1(\tilde{m}_n) \le \liminf_{i \to \infty} E_{h_i}(m^{(i)}).
\]
Set $d_n = \deg(\tilde{m}_n)$. 
Then $$\int_{-R}^R (m_n^{(i)})'\cdot (m_n^{(i)})^\perp \, dx_1\stackrel{i\to \infty}{\to} \int_{-R}^R \tilde m_n'\cdot \tilde m_n^\perp \, dx_1= \tilde \phi_n(R)-\tilde \phi_n(-R)=2\pi d_n+o(1)$$
as $R\to \infty$, where $\tilde \phi_n$ is a lifting of $\tilde m_n$. In other words, the degree carried by $m_n^{(i)}$ (corresponding of the group of transitions near $a_n^{(i)}$) is asymptotically given by $d_n$ as $i$ becomes very large. For fixed $i$, the degrees of each group of transitions of $m_n^{(i)}$ are a partition of $\ell-\alpha_i/\pi$, so we infer, letting $i\to \infty$, that $\sum_{n = 1}^N d_n = \ell$ (as $\alpha_i\to 0$) and $d_n\geq 0$. Some of these degrees $d_n$ may be zero; this is the case if, and only if, the group of transitions near $a_n^{(i)}$ has degree $\alpha_i/\pi$, which does \emph{not} apply to $n = 1$ and $n = N$ by Lemma \ref{lem:no_unaccounted_walls}. We eliminate those $n$, so that after relabelling the indices, we may assume that $(d_1, \dotsc, d_N)$ is a nontrivial partition of $\ell$.

Finally, the upper semicontinuity of Lemma \ref{lem:upper_semicontinuity}
implies that
\[
\E_1(\ell) \ge \liminf_{i \to \infty} \E_{h_i}(\ell - \alpha_i/\pi) = \liminf_{i \to \infty} E_{h_i}(m^{(i)}) \ge \sum_{n = 1}^N E_1(\tilde{m}_n) \ge \sum_{n = 1}^N \E_1(d_n).
\]

On the other hand, in the proof of Theorem \ref{thm:existence}, we have
seen that
\[
\E_1(\ell) < \sum_{n = 1}^N \E_1(d_n)
\]
for any nontrivial partition $(d_1, \dotsc, d_N)$ of $\ell$.
This contradiction concludes the proof.
\end{proof}

\begin{remark}
The above argument also proves for the case $h=1$ that for every $\ell\in \Z$, there exists a constant 
$\Lambda_\ell>0$ such that every minimizer $m$ of $E_1$ over the set $\A_1(\ell)$ has the property
\[
\diam \set{x_1 \in \R}{m(x_1) = (\pm 1, 0)} \le \Lambda_\ell.
\]
\end{remark}

The following lemma shows that the behaviour of a minimiser is consistent with the
statement of Theorem \ref{thm:structure} at least at the tails.

\begin{lemma} \label{lem:behaviour_at_tails}
Let $\ell \in \N$. Then there exists $H \in (0, 1)$ such that any $h \in (H, 1]$
has the following property. Suppose that $m \in \A_h(\ell - \alpha/\pi)$ minimises
$E_h$ in $\A_h(\ell - \alpha/\pi)$. If $a \in \R$ is such that $m_1(x_1) > -1$ for all $x_1 > a$,
then $m_1(x_1) \le h$ for all $x_1 > a$.
\end{lemma}

\begin{proof}
If $h=1$, this is obvious. If $h<1$, the arguments are similar to the proof of Theorem \ref{thm:non-existence}. Without loss of generality we may
assume that $m_1(a) = -1$. We define
\[
m_1^+(x_1) = \begin{cases}
m_1(x_1) & \text{if $x_1 > a$ and $m_1(x_1) > h$}, \\
h & \text{otherwise},
\end{cases}
\]
and
\[
m_1^-(x_1) = \begin{cases}
m_1(x_1) & \text{if $x_1 \le a$ or $m_1(x_1) \le h$}, \\
h & \text{otherwise}.
\end{cases}
\]
Then this gives the decomposition $m_1 = m_1^+ + m_1^- - h$, and $(m_1^- - h)(m_1^+ - h)=0$ everywhere. Assume for contradiction that $m_1^+ \not\equiv h$. Then we denote by $b$ the minimum of the support of $m_1^+ - h$ (so $b>a$) and let $[c,b]$ be the connected component of the support of $m_1^- - h$ containing $b$. Then the same estimates as in the proof of Theorem \ref{thm:non-existence} apply.
Provided that $1 - h$ is sufficiently small, we can then construct $m_2^- \colon \R \to [-1,1]$ such that
$m^- = (m_1^-, m_2^-) \in \A_h(\ell -\alpha/\pi)$ satisfies $E_h(m^-) < E_h(m)$. But this is impossible,
as $m$ is assumed to be a minimiser.
\end{proof}

For the proof of Theorem \ref{thm:structure}, we also need the following H\"older estimate
for the derivatives of minimisers of $E_h$.

\begin{lemma} \label{lem:bound_for_derivative}
Let $\ell \in \N$. Then there exists a constant $C$ such that for all $h \in [0, 1]$,
any minimiser $m$ of $E_h$ in $\A_h(\ell - \alpha/\pi)$ satisfies
$|m'(s) - m'(t)| \le C\sqrt{|s - t|}$ for all $s, t \in \R$.
\end{lemma}

\begin{proof}
If we write $m = (\cos \phi, \sin \phi)$, then we have the Euler-Lagrange equation
\eqref{eqn:Euler-Lagrange_higher}. This implies that
\[
|\phi''| \le 2 + |u'(\blank, 0)|.
\]
As $u'(\cdot, 0)$ is the Dirichlet-to-Neumann operator associated to $m_1-h$, see \cite[Section 1.6]{Ignat-Moser:17}, we have
\[
\|u'(\blank, 0)\|_{L^2(\R)} = \|m_1'\|_{L^2(\R)} \le \sqrt{2E_h(m)}.
\]
Thus
\[
\|\phi''\|_{L^2(a, a + 1)} + \|\phi'\|_{L^2(a, a + 1)} \le 2 + 2\sqrt{2E_h(m)}
\]
for any $a \in \R$.
Finally, the right-hand side is bounded by a constant depending only on $\ell$ by Lemma \ref{lem:upper_semicontinuity}.
The Sobolev embedding theorem then implies the desired inequality.
\end{proof}

\begin{proof}[Proof of Theorem \ref{thm:structure}]
We know by Lemma \ref{lem:no_unaccounted_walls} that for any minimiser $m$ of $E_h$
in $\A_h(\ell - \alpha/\pi)$, there exist $a_1, \dotsc, a_{2\ell - 1} \in \R$
with $a_1 < \dotsb < a_{2\ell - 1}$ such that
\[
m_1(a_n) = (-1)^n, \quad n = 1, \dotsc, 2\ell - 1.
\]
Lemma \ref{lem:behaviour_at_tails} implies that $m_1 \le h$ in $(-\infty, a_1)$
and in $(a_{2\ell - 1}, \infty)$. Thus it suffices to examine the behaviour in
the intervals $(a_n, a_{n + 1})$ for $n = 1, \dotsc, 2\ell - 2$.

Fix $n \in \{1, \dotsc, 2\ell - 2\}$. Without loss of generality we may assume
that $m_1(a_n) = 1$ and $m_1(a_{n + 1}) = -1$. We need to show that there exists
$b_n \in (a_n, a_{n + 1})$ such that $m_1 \ge h$ in $[a_n, b_n]$ and $m_1 \le h$
in $[b_n, a_{n + 1}]$. To this end, define
\[
b_n = \inf \set{x_1 \in (a_n, a_{n + 1})}{m_1(x_1) \le h}.
\]
It suffices to show that $m_1 \le h$ in $(b_n, a_{n + 1})$. We argue by contradiction
and assume that there exist $c_n, c_n' \ge b_n$ with $c_n' > c_n$ such that
$m_1(c_n) = m_1(c_n') = h$ but $m_1 > h$ in $(c_n, c_n')$.

Let
\begin{align*}
\hat{m}_1^+(x_1) & = \begin{cases}
m_1(x) & \text{if $x_1 \not\in (c_n, c_n')$ and $m_1(x_1) > h$}, \\
h & \text{otherwise},
\end{cases} \\
\tilde{m}_1(x_1) & = \begin{cases}
m_1(x) & \text{if $x_1 \in (c_n, c_n')$}, \\
h & \text{otherwise},
\end{cases}
\end{align*}
and $\hat{m}_1^- = \min\{m_1, h\}$. Then $m_1 = \hat{m}_1^+ + \tilde{m}_1 + \hat{m}_1^- - 2h$.
Furthermore, define
\begin{align*}
\hat{m}(x_1) & = \begin{cases}
m(x_1) & \text{if $x_1 \not\in (c_n, c_n')$}, \\
m(c_n) & \text{if $x_1 \in (c_n, c_n')$},
\end{cases} \\
\tilde{m}(x_1) & = \begin{cases}
m(x_1) & \text{if $x_1 \in (c_n, c_n')$}, \\
m(c_n) & \text{if $x_1 \not\in (c_n, c_n')$}.
\end{cases}
\end{align*}
Then
\[
\begin{split}
E_h(m) & = E_h(\hat{m}) + E_h(\tilde{m}) + \scp{\hat{m}_1}{\tilde{m}_1}_{\dot{H}^{1/2}(\R)} \\
& = E_h(\hat{m}) + E_h(\tilde{m}) + \scp{\hat{m}_1^+}{\tilde{m}_1}_{\dot{H}^{1/2}(\R)} + \scp{\hat{m}_1^-}{\tilde{m}_1}_{\dot{H}^{1/2}(\R)} \\
& = E_h(\hat{m}) + E_h(\tilde{m}) - \frac{1}{\pi} \int_{-\infty}^\infty (\tilde{m}_1(s) - h) \int_{-\infty}^\infty \frac{\hat{m}_1^+(t) + \hat{m}_1^-(t) - 2h}{(s - t)^2} \, dt \, ds.
\end{split}
\]
But since $m$ is a minimiser of $E_h$ for its winding number (which coincides with the
winding number of $\hat{m}$), and since $E_h(\tilde{m}) > 0$ by our assumption,
it follows that
\[
\int_{-\infty}^\infty (\tilde{m}_1(s) - h) \int_{-\infty}^\infty \frac{\hat{m}_1^+(t) + \hat{m}_1^-(t) - 2h}{(s - t)^2} \, dt \, ds > 0.
\]

Next we want to show that in fact, if $1 - h$ is sufficiently small, then
\begin{equation} \label{eqn:cross-terms}
\int_{-\infty}^\infty \frac{\hat{m}_1^+(t) + \hat{m}_1^-(t) - 2h}{(s - t)^2} \, dt \le 0
\end{equation}
for every $s \in (c_n, c_n')$. As $\tilde{m}_1(s) - h > 0$ in $(c_n, c_n')$ and
$\tilde{m}_1(s) - h = 0$ outside of $(c_n, c_n')$ by construction, this will give the desired
contradiction.

In order to prove \eqref{eqn:cross-terms}, let $C_1$ be the constant from
Lemma \ref{lem:bound_for_derivative}. By the choice of $c_n'$, we know that $m_1'(c_n') \le 0$.
For $x_1 > c_n'$, we conclude that $m_1'(x_1) \le C_1(x_1 - c_n')^{1/2}$. Integration then
yields the inequality $m_1(x_1) - h \le \frac{2C_1}{3}(x_1 - c_n')^{3/2}$ for $x_1 > c_n'$.
Thus if $t > c_n'$ is such that $m_1(t) \ge h$, then
\[
\hat{m}_1^+(t) - h \le \min\{C_2(t - c_n')^{3/2}, 1 - h\},
\]
where $C_2 = 2C_1/3$.
Given $s \in (c_n, c_n')$, we then see that $t-s\geq t-c_n'$ for $t>c_n'$ and
\[
\begin{split}
\int_{c_n'}^\infty \frac{\hat{m}_1^+(t) - h}{(s - t)^2} \, dt & \le C_2 \int_{c_n'}^{c_n' + ((1 - h)/C_2)^{2/3}} \frac{(t - c_n')^{3/2}}{(s - t)^2} \, dt + (1 - h)\int_{c_n' + ((1 - h)/C_2)^{2/3}}^{\infty} \frac{dt}{(s - t)^2} \\
& \le C_2 \int_0^{((1 - h)/C_2)^{2/3}} \frac{dt}{\sqrt{t}} + (1 - h) \int_{((1 - h)/C_2)^{2/3}}^\infty \frac{dt}{t^2} \\
& = 3C_2^{2/3} (1 - h)^{1/3}.
\end{split}
\]
Similarly, as $m_1'(c_n)\geq 0$, then $-m_1'(x_1)\leq C_1 (c_n-x_1)^{1/2}$ for every $x_1<c_n$, so that 
$m_1(x_1) - h \le \frac{2C_1}{3}(c_n-x_1)^{3/2}$ for $x_1 < c_n$. Therefore,
\[
\int_{-\infty}^{c_n} \frac{\hat{m}_1^+(t) - h}{(s - t)^2} \, dt \le 3C_2^{2/3} (1 - h)^{1/3}.
\]

On the other hand, Lemma \ref{lem:bound_for_derivative} also implies that there exists
a constant $R > 0$, depending only on $\ell$, such that $m_1 \le -\frac{1}{2}$
in $[a_n - R, a_n + R]$ for any odd number $n$. If $\Lambda$ is the constant from
Proposition \ref{prop:profile_width}, then this implies that
\[
\int_{-\infty}^\infty \frac{\hat{m}_1^-(t) - h}{(s - t)^2} \, dt \le \int_{a_1 - R}^{a_1 + R} \frac{\hat{m}_1^-(t) - h}{(s - t)^2} \, dt \le -\frac{1}{2} \int_{a_1 - R}^{a_1 + R} \frac{dt}{(s - t)^2} \le -\frac{R}{(\Lambda+R)^2}
\]
because $s\in (c_n, c_n')\subset (a_n, a_{n+1})$.
Therefore,
\[
\int_{-\infty}^\infty \frac{\hat{m}_1^+(t) + \hat{m}_1^-(t) - 2h}{(s - t)^2} \, dt \le 6C_2^{2/3} (1 - h)^{1/3} - \frac{R}{(\Lambda+R)^2}.
\]
If $1 - h$ is sufficiently small, then the right-hand side is negative, which
proves \eqref{eqn:cross-terms}. Thus the estimate gives
the desired contradiction and concludes the proof.
\end{proof}

\appendix

\section{Technical Lemmas}

Here we give a few auxiliary results that are required for our proofs but
are not specific to our problem.

\begin{lemma} \label{lem:auxiliary_inequality}
For any $\alpha \in [0, \frac{\pi}{2}]$,
\[
\inf_{\phi \in (0, \pi)} \frac{1 - \cos \alpha \cos \phi}{\sin^2 \phi} = \frac{1}{2} (1 + \sin \alpha)\geq \frac12.
\]
\end{lemma}

\begin{proof}
For $\alpha = \frac{\pi}{2}$, this is obvious. For $\alpha = 0$, it follows
from the observation that
\[
\frac{1 - \cos \phi}{\sin^2 \phi} = \frac{1}{1 + \cos \phi}.
\]
We now assume that $0 < \alpha < \frac{\pi}{2}$.
Consider the function $f \colon (0, \pi) \to \R$ defined by
\[
f(\phi) = \frac{1 - \cos \alpha \cos \phi}{\sin^2 \phi}, \quad \phi\in (0, \pi).
\]
Then clearly $\lim_{\phi \searrow 0} f(\phi) = \lim_{\phi \nearrow \pi} f(\phi) = \infty$.
We compute
\[
f'(\phi) = \frac{\cos \alpha \sin^2 \phi - 2 \cos \phi (1 - \cos \alpha \cos \phi)}{\sin^3 \phi},  \quad \phi\in (0, \pi).
\]
At any zero $\phi$ of $f'$, we have the identity
\[
\begin{split}
0 & = \cos \alpha \sin^2 \phi - 2 \cos \phi + 2\cos \alpha \cos^2 \phi \\
& = \cos \alpha \cos^2 \phi - 2 \cos \phi + \cos \alpha.
\end{split}
\]
Regarding this as a quadratic equation in $\cos \phi$, we see that
\[
\cos \phi = \frac{1}{\cos \alpha} \pm \tan \alpha.
\]
But only one of these is in $[-1, 1]$, and it follows that $\cos \phi = \frac{1}{\cos \alpha} - \tan \alpha$.
Therefore, the function $f'$ has only one zero and this is the unique minimum of $f$.
At this point, we compute that $f(\phi) = \frac{1}{2}(1 + \sin \alpha)$.
\end{proof}

\begin{lemma} \label{lem:L^2-estimate_for_v}
Let $p \in (1, 2]$ and $f \in L^p(\R) \cap \dot{H}^{1/2}(\R)$. Furthermore, let $v \in \dot{H}^1(\R_+^2)$ be
the unique solution of $\Delta v = 0$ in $\R_+^2$ with $v(\blank, 0) = f$ in $\R$. Let $R > 0$. Then
\[
\int_0^R \int_{-\infty}^\infty v^2 \, dx_1 \, dx_2 \le \left(\frac{8p}{p + 2}\right)^{3 - 2/p} \frac{pR^{2 - 2/p}}{2\pi^2 (p - 1)} \|f\|_{L^p(\R)}^2.
\]
Furthermore, if $f \in L^1(\R)$ and $I_R=(\frac1R, R)$ (for $R\geq 1$) or $I_R=(R, \frac1R)$ (for $R\leq 1$), then
\[
\int_{I_R}\int_{-\infty}^\infty v^2 \, dx_1 \, dx_2 \le \frac{16}{3\pi^2} |\log R| \|f\|_{L^1(\R)}^2,
\]
and if $f \in L^\infty(\R)$, then
\[
\int_0^{1/R} \int_{-R}^R v^2 \, dx_1 \, dx_2 \le 2 \|f\|_{L^\infty(\R)}^2.
\]
\end{lemma}

\begin{proof}
By the Poisson formula,
\[
v(x_1, x_2) = \frac{x_2}{\pi} \int_{-\infty}^\infty \frac{f(t)}{(t - x_1)^2 + x_2^2} \, dt.
\]
If we write
\[
g_{x_2}(t) = \frac{x_2}{\pi(t^2 + x_2^2)},
\]
then this becomes $v(x_1, x_2) = (f * g_{x_2})(x_1)$.

We allow also the case $p = 1$ for the moment, so $p \in [1, 2]$.
Set $q = \frac{2p}{3p - 2}\in [1,2]$, so that
$\frac{1}{p} + \frac{1}{q} = \frac{3}{2}$. Then Young's convolution inequality implies that
\[
\left\|f * g_{x_2}\right\|_{L^2(\R)} \le \|f\|_{L^p(\R)} \|g_{x_2}\|_{L^q(\R)}.
\]
We estimate
\[
\begin{split}
\int_{-\infty}^\infty \left|g_{x_2}(t)\right|^q \, dt & =  \frac{2x_2^q}{\pi^q} \int_0^\infty \frac{dt}{(t^2 + x_2^2)^q} \\
& \le \frac{2x_2^q}{\pi^q} \left(\int_0^{x_2} \frac{dt}{x_2^{2q}} + \int_{x_2}^\infty \frac{dt}{t^{2q}}\right) \\
& = \frac{4q x_2^{1 - q}}{\pi^q (2q - 1)} = \frac{8p x_2^{1 - q}}{\pi^q (p + 2)}.
\end{split}
\]
Let $C_1 = \frac{8p}{\pi^q (p + 2)}$. We now treat the cases $p \in (1, 2]$
and $p = 1$ separately.

\begin{severalcases}

\case $p \in (1, 2]$. Then $q < 2$. For any $R > 0$, we have
\[
\begin{split}
\int_0^R \int_{-\infty}^\infty v^2 \, dx_1 \, dx_2 & \le \|f\|_{L^p(\R)}^2 \int_0^R \|g_{x_2}\|_{L^q(\R)}^2 \, dx_2 \\
& \le C_1^{2/q} \|f\|_{L^p(\R)}^2 \int_0^R x_2^{2/q - 2} \, dx_2 \\
& = C_1^{2/q} \frac{qR^{2/q - 1}}{2 - q} \|f\|_{L^p(\R)}^2 \\
& = C_1^{3 - 2/p} \frac{pR^{2 - 2/p}}{2p - 2} \|f\|_{L^p(\R)}^2.
\end{split}
\]
This is the first inequality of the lemma.

\case $p = 1$. Then
\[
\begin{split}
\int_{I_R} \int_{-\infty}^\infty v^2 \, dx_1 \, dx_2 & \le \int_{I_R}  \|f\|_{L^1(\R)}^2 \|g_{x_2}\|_{L^2(\R)}^2 \, dx_2 \\
& \le \frac{8}{3\pi^2} \|f\|_{L^1(\R)}^2 \int_{I_R}  \frac{dx_2}{x_2} \\
& = \frac{16}{3\pi^2} \|f\|_{L^1(\R)}^2 |\log R|,
\end{split}
\]
which proves the second inequality.
\end{severalcases}

Finally, the third inequality in the lemma is an obvious consequence of the maximum principle, which implies that $\sup_{\R^2_+}|v|\leq \sup_{\R}|f|$.
\end{proof}

\begin{lemma} \label{lem:decay_implies_L^1_estimate}
Let $\sigma > 1$. Suppose that $\psi \colon (1, \infty) \to [0, \infty)$ is an integrable
function such that
\[
\int_R^\infty \psi^2 \, dt \le R^{-\sigma}
\]
for all $R \ge 1$. Then for any $R \ge 1$,
\[
\int_R^\infty \psi \, dt \le \frac{2\sqrt{\sigma R^{1 - \sigma}}}{\sigma - 1}.
\]
\end{lemma}

\begin{proof}
Let $\omega \in (1, \sigma)$. We estimate
\[
\begin{split}
\int_R^\infty \psi \, dt & \le \left(\int_R^\infty t^{-\omega} \, dt\right)^{1/2} \left(\int_R^\infty t^\omega \psi^2 \, dt\right)^{1/2} \\
& = \left(\frac{\omega R^{1 - \omega}}{\omega - 1} \int_R^\infty \left(\int_R^t s^{\omega - 1} \, ds + \frac{R^\omega}{\omega}\right) (\psi(t))^2 \, dt\right)^{1/2} \\
& = \left(\frac{\omega R^{1 - \omega}}{\omega - 1} \left(\int_R^\infty \int_s^\infty (\psi(t))^2 \, dt \, s^{\omega - 1} \, ds + \frac{R^\omega}{\omega} \int_R^\infty \psi^2 \, dt\right)\right)^{1/2} \\
& \le \left(\frac{\omega R^{1 - \omega}}{\omega - 1} \left(\int_R^\infty s^{\omega - \sigma - 1} \, ds + \frac{R^{\omega -\sigma}}{\omega}\right) \right)^{1/2} \\
& = \left(\frac{\sigma R^{1 - \sigma}}{(\sigma - \omega)(\omega - 1)}\right)^{1/2}.
\end{split}
\]
Here we have used H\"older's inequality, Fubini's theorem, and the inequality from the hypothesis.
Choosing $\omega = \frac{1}{2}(\sigma + 1)$ finally gives the inequality stated.
\end{proof}

\bibliographystyle{amsplain}
\bibliography{bib-3}

\end{document}